\documentclass[psamsfonts, reqno]{amsart}

\usepackage{amssymb,amsfonts,amscd}
\usepackage[all,arc]{xy}
\usepackage{enumerate, bbm}
\usepackage{mathrsfs}
\usepackage{mathtools}
\usepackage{hyperref}
\usepackage[usenames, dvipsnames]{color}
\usepackage{graphicx}
\usepackage{enumitem} 
\graphicspath{ {TexImages/} }

\newtheorem{thm}{Theorem}[section]

\newtheorem{prop}[thm]{Proposition}
\newtheorem{lem}[thm]{Lemma}
\newtheorem{claim}[thm]{Claim}
\theoremstyle{definition}

\newtheorem{strat}[thm]{Strategy}
\newtheorem{quest}[thm]{Question}
\newtheorem{conj}[thm]{Conjecture}\theoremstyle{remark}

\setlist[enumerate]{itemsep=2ex, topsep=2ex} 
\setlist[itemize]{itemsep=2ex, topsep=2ex}
\renewcommand{\c}[1]{\mathcal{#1}}

\renewcommand{\r}{\right}

\newcommand{\half}{\frac{1}{2}}
\newcommand{\quart}{\frac{1}{4}}

\newcommand{\ol}[1]{\overline{#1}}

\newcommand{\rec}[1]{\frac{1}{#1}}
\newcommand{\f}[2]{\frac{#1}{#2}}
\newcommand{\gam}{\gamma}
\newcommand{\sm}{\setminus}

\newcommand{\sub}{\subseteq}

\newcommand{\floor}[1]{\lfloor #1\rfloor}

\DeclareMathOperator*{\ex}{ex}

\newcommand{\sat}{\mathrm{sat}}

\hypersetup{
colorlinks,
citecolor=blue,
filecolor=blue,
linkcolor=blue,
urlcolor=blue,
linktocpage
}

\title{Saturation Games for Odd Cycles}

\author{Sam Spiro}
\date{\today}
\begin{document}

\begin{abstract}
	Given a family of graphs $\mathcal{F}$, we consider the $\mathcal{F}$-saturation game.  In this game, two players alternate adding edges to an initially empty graph on $n$ vertices, with the only constraint being that neither player can add an edge that creates a subgraph that lies in $\mathcal{F}$.  The game ends when no more edges can be added to the graph.  One of the players wishes to end the game as quickly as possible, while the other wishes to prolong the game.  We let $\textrm{sat}_g(\mathcal{F};n)$ denote the number of edges that are in the final graph when both players play optimally.
	
	The $\{C_3\}$-saturation game was the first saturation game to be considered, but as of now the order of magnitude of $\textrm{sat}_g(\{C_3\},n)$ remains unknown.  We consider a variation of this game.  Let $\mathcal{C}_{2k+1}:=\{C_3,\ C_5,\ldots,C_{2k+1}\}$. We prove that $\textrm{sat}_g(\mathcal{C}_{2k+1};n)\ge(\frac{1}{4}-\epsilon_k)n^2+o(n^2)$ for all $k\ge 2$ and that $\textrm{sat}_g(\mathcal{C}_{2k+1};n)\le (\frac{1}{4}-\epsilon'_k)n^2+o(n^2)$ for all $k\ge 4$, with $\epsilon_k<\frac{1}{4}$ and $\epsilon'_k>0$ constants tending to 0 as $k\to \infty$.  In addition to this we prove $\textrm{sat}_g(\{C_{2k+1}\};n)\le \frac{4}{27}n^2+o(n^2)$ for all $k\ge 2$, and $\textrm{sat}_g(\mathcal{C}_\infty\setminus C_3;n)\le 2n-2$, where $\mathcal{C}_\infty$ denotes the set of all odd cycles.
\end{abstract}
	\maketitle
\section{Introduction}
Hajnal proposed the following game.  Initially $G$ is an empty graph on $n$ vertices.  Two players alternate turns adding edges to $G$, with the only restriction being that neither player is allowed to add an edge that would create a triangle.  The last player to add an edge wins the game, and the central question is which player wins this game as a function of $n$.

The answer to this problem is known only for small values of $n$, the most recent result being $n=16$ by Gordinowicz and Pra\l{}at \cite{pralat}.  A variation of this game was considered by F\"{u}redi, Reimer, and Seress \cite{furedi}.  In the modified version of the game, there are two players, Max and Mini, who alternate turns adding edges to an initially empty graph on $n$ vertices with the same rules as in Hajnal's original triangle-free game.  The main difference is that once no more edges can be added to $G$, Max receives a point for every edge in the graph and Mini loses a point for every edge in the graph, with both players trying maximize the number of points they receive at the end.  The question is now to figure out how many edges are at the end of the game when both players play optimally.

This game can be generalized.  For a family of graphs $\c{F}$, we say that a graph $G$ is $\c{F}$-saturated if $G$ contains no graph of $\c{F}$ as a subgraph, but adding any edge to $G$ would create a subgraph of $\c{F}$.  The $\c{F}$-saturation game consists of two players, Max and Mini, who alternate turns adding edges to an initially empty graph $G$ on $n$ vertices, with the only restriction being that $G$ is never allowed to contain a subgraph that lies in $\c{F}$.  The game ends when $G$ is $\c{F}$-saturated.  The payoff for Max is the number of edges in $G$ when the game ends, and Mini's payoff is the opposite of this.  We let $\sat_g(\c{F};n)$ denote the number of edges that the graph in the $\c{F}$-saturation game ends with when both players play optimally, and we call this quantity the game $\c{F}$-saturation number.  

We note that this game, and hence the value of $\sat_g(\c{F};n)$, depends on whether Max or Mini makes the first move of the game, and in general this choice can significantly affect the value of $\sat_g(\c{F};n)$, as is illustrated in \cite{hefetz}.  For simplicity we will only consider the game where Max makes the first move, though we claim that all of our results continue to hold when Mini makes the first move by making small adjustments to our current proofs.

Let $C_k$ denote the cycle of length $k$.  The $\{C_3\}$-saturation game was the original saturation game studied in \cite{furedi}, where they proved what is still the best known lower bound of  $\half n\log n+o(n\log n)$ for $\sat_g(\{C_3\};n)$.  Erd\H{o}s claimed to have proved an upper bound of $n^2/5$ for this game, but this proof has been lost. Recently, Bir{\'o}, Horn, and Wildstrom \cite{horn} published the first non-trivial asymptotic upper bound of $\f{26}{121}n^2+o(n^2)$ for $\sat_g(\{C_3\};n)$.  A number of other results have been obtained for specific choices of $\c{F}$, see for example \cite{westSurv}, \cite{westMatch}, and \cite{lee1}.  In addition to this, saturation games have recently been generalized to directed graphs \cite{lee2}, hypergraphs \cite{patkos}, and to avoiding more general graph properties such as $k$-colorability in \cite{hefetz} and \cite{keusch}.

\subsection{Main Results}\leavevmode
 
Let $\c{C}_{2k+1}:=\{C_3,C_5,\ldots,C_{2k+1}\}$, and let $\c{C}_\infty$ denote the set of all odd cycles.  Most of this paper will be focused on studying the $\c{C}_{2k+1}$-saturation games for $k\ge 2$.  The key idea with these games is that by forbidding either player from making $C_{5}$'s, both players can utilize a strategy that keeps the graph essentially bipartite throughout the game.  This makes it significantly easier to analyze the correctness of our proposed strategies, and to bound the number of edges that are in the final graph. Our main result is the following upper and lower bound for $\sat_g(\c{C}_{2k+1};n)$ and most values of $k$.

\begin{thm}\label{T-gen}
	For $k\ge 4$, 
	\[
		\left(\quart-\f{1}{4k^2}\right)n^2+o(n^2) \le \sat_g(\c{C}_{2k+1};n)\le \left(\quart -\rec{ 20^6k^4}\r)n^2+o(n^2).
	\]
\end{thm}

We can also obtain a quadratic lower bound for smaller values of $k$, and more generally for any collection of non-bipartite graphs which contains $C_3$ and $C_5$.

\begin{thm}\label{T-5}
	If $\c{C}$ is a set of graphs which are not bipartite with $C_3,C_5\in \c{C}$,  then \[\sat_g(\c{C};n)\ge \f{6}{25}n^2+o(n^2).\]
\end{thm}

We emphasize that these results do not imply a quadratic lower bound for the triangle-free game.  We consider two more saturation games.  The first is the game where only one odd cycle is forbidden.
\begin{thm}\label{T-OneCyc}
	For $k\ge 2$, \[\sat_g(\{C_{2k+1}\};n)\le \rec{12}\left(1+\rec{\ell}\right)^2n^2+o(n^2),\]
	where $\ell=\max(3,\floor{\sqrt{2k}})$.  In particular, $\sat_g(\{C_{2k+1}\};n)\le \f{4}{27}n^2+o(n^2)$ for all $k\ge 2$.
\end{thm}

We also consider the ``complement'' of the $\{C_3\}$-saturation game where every odd cycle except $C_3$ is forbidden.  It turns out that in this setting the game saturation number is linear.
\begin{thm}\label{T-Mod3}
	\[ \sat_g(\c{C}_\infty\setminus \{C_3\};n)\le 2n-2.\]
\end{thm}

This result is in sharp contrast to the fact that $\sat_g(\c{C}_\infty;n)=\floor{\quart n^2}$, see \cite{westSurv}.

\textbf{Notation}.  Throughout the paper we let $G^t$ denote the graph in the relevant saturation game after $t$ edges have been added, and we let $e^t$ denote the edge of $G^t$ that is not in $G^{t-1}$.  We let $N^t(x)$ denote the neighborhood of $x$ in $G^t$ and let $d^t(x,y)$ denote the distance between $x$ and $y$ in $G^t$.  We let $t=\infty$ correspond to the point in time when the graph has become $\c{F}$-saturated.  If $X^t$ is a real number depending on $t$, we define $\Delta(X^t)=X^t-X^{t-2}$.  We let $E(G)$ denote the set of edges of the graph $G$ and let $e(G)=|E(G)|$.  We write $G-X$ when $X$ is a vertex, edge, or set of vertices and edges to denote the graph obtained by deleting these vertices and edges from $G$.  We omit floor and ceiling signs throughout whenever these are not crucial.

\textbf{Organization}.  In Section 2 we present a strategy for Max that guarantees that the game ends with at least as many edges as stated in Theorem~\ref{T-5}.  In Section 3 we modify this strategy to obtain the lower bound of Theorem~\ref{T-gen}.  In Section 4 we present a strategy for Mini that guarantees that the game ends with at most as many edges as the upper bound of Theorem~\ref{T-gen}.  Theorem~\ref{T-OneCyc} is proven in Section 5.  Theorem~\ref{T-Mod3} is proven in Section 6.  We end with some concluding remarks in Section 7.
\section{Proof of Theorem~\ref{T-5}}\label{S-Low1}

Let $\c{C}$ be as in the hypothesis of Theorem~\ref{T-5}.  We wish to construct a strategy for Max in the $\c{C}$-saturation game such that at the end of each of Max's turns, $G^t$ is bipartite with parts of roughly the same size. To this end, let $uv$ denote the edge of $G^1$.  Let $1<\gamma\le 2$ and $\delta=\rec{\gamma-1}$.  We say that $G^t$ is $\gamma$-good if it satisfies the following four conditions.

\begin{itemize}
	\item[(1*)] $G^t$ contains exactly one non-trivial connected component, and this component is bipartite with parts $U^t\ni u$ and $V^t\ni v$.
	
	Let $U_0^t=N^t(v)$ (the good vertices), and $U_1^t=U^t\setminus U_0^t$ (the bad vertices).  Define an analogous partition for $V^t$.
	
	\item[(2*)] Every vertex of $U^t\cup V^t$ is adjacent to a vertex in $U_0^t\cup V_0^t$.
	
	\item[(3*)] $b_U^t:=|V_1^t|+(|U^t|-\gamma|V^t|-\delta)\le0$ and $b_V^t:=|U_1^t|+(|V^t|-\gamma|U^t|-\delta)\le 0$.
	
	\item[(4*)] $b_U^t+b_V^t\le -2$.
	
\end{itemize}
We note that $b_U^t\le 0$ implies that, up to an additive constant factor, $|U^t|$ is larger than $|V^t|$ by a multiplicative factor of at most $\gam$.  Moreover, if $|U^t|\approx \gam |V^t|$, then $b_U^t\le 0$ guarantees that there are few vertices in $V_1^t$.  We note that (2*) and (4*) are trivially satisfied if $U_1^t=V_1^t=\emptyset$.  An important consequence of being $\gam$-good is the following.  To make this statement precise, if $G^t$ satisfies (1*) but $G^{t+1}$ contains more than one non-trivial connected component, we define $U^{t+1}:=U^t$ and $V^{t+1}:=V^t$.

\begin{lem}\label{L-Bipartite}
	Let $\c{C}$ be a set of graphs with $C_3,C_5\in \c{C}$.  Let $t$ be such that $G^{t}$ satisfies (1*) and (2*).  Then $U^{t+1}$ and $V^{t+1}$ are independent sets for any valid choice of $e^{t+1}$ in the $\c{C}$-saturation game. 
\end{lem}
\begin{proof}
	$U^{t}$ and $V^{t}$ are independent sets since $G^{t}$ satisfies (1*).  If $v',v''\in V^{t}$, let $u',u''\in U_0^{t}$ be neighbors of $v'$ and $v''$ respectively, noting that such vertices exist since $G^{t}$ satisfies (2*). Then \[d^{t}(v',v'')\le d^{t}(v',u')+d^{t}(u',v)+d^{t}(v,u'')+d^{t}(u'',v'')=4.\]  Thus having $e^{t+1}=v'v''$ would create either a $C_3$ or a $C_5$ since $d^{t}(v',v'')$ is even, which is forbidden in the $\c{C}$-saturation game.  The analysis for $U^{t+1}$ is similar.
\end{proof}

We derive Theorem~\ref{T-5} by first proving the following.
\begin{prop}\label{P-algLow}
	Let $\c{C}$ be a set of graphs which are not bipartite and with $C_3,C_5\in \c{C}$.  Then there exists a strategy for Max in the $\c{C}$-saturation game such that for all odd $t$, whenever $G^{t-1}$ contains an isolated vertex, Max can choose $e^t$ so that $G^{t}$ is $\f{3}{2}$-good.
\end{prop}

\begin{proof}
Throughout this proof we implicitly use the fact that since $\c{C}$ consists of graphs which are not bipartite, any move that Max makes which keeps the graph bipartite must be legal. It is not difficult to see that $G^1$ is $\f{3}{2}$-good. Assume Max has been able to play so that $G^{t-2}$ is $\f{3}{2}$-good with $t$ odd.  If $G^{t-1}$ contains no isolated vertices then we are done, so assume that there exists an isolated vertex $z$ in $G^{t-1}$.  Let $e^{t-1}=xy$.  We will say that $e^{t-1}$ is an $I$ (Internal) move if $x\in U^{t-2},\ y\in V^{t-2}$, an $O$ (Outside) move if $x,y\notin U^{t-2}\cup V^{t-2}$, an $AU$ (Add to $U$) move if $x\in V^{t-2},\ y\notin U^{t-2}\cup V^{t-2}$, and an $AV$ (Add to $V$) move if $x\in U^{t-2},\ y\notin U^{t-2}\cup V^{t-2}$.  Note that an $AU$ move causes $y$ to be added to $U^{t-1}$.    Lemma~\ref{L-Bipartite} shows that $e^{t-1}$ must be one of the four types of moves discussed above (possibly after relabeling $x$ and $y$), so it is enough to show how Max reacts to each of these types of moves.  

We note that if we assume $G^{t-2}$ satisfies (1*), any vertex not in $U^{t-2}\cup V^{t-2}$ must be isolated. When Max plays, it will always be obvious that (1*) is maintained, so we will not verify this condition in our analysis.  Throughout the rest of this section we write $\gam$ instead of $\f{3}{2}$ whenever our argument continues to hold when $G^{t-2}$ is assumed to be $\gam$-good for any $1<\gam\le 2$, and we will emphasize whenever we need to use $\gam=\f{3}{2}$ in our proofs.  This will make proving the lower bound of Theorem~\ref{T-gen} somewhat simpler.

\begin{claim}\label{Cl-I}
	If $e^{t-1}$ is an $I$ move, then Max can play so that $G^t$ is $\gamma$-good.
\end{claim}
\begin{proof}
	If there exists $u'\in U^{t-1},\ v'\in V^{t-1}$ with $u'v'\notin G^{t-1}$, then Max adds the edge $u'v'$, and it is not hard to see that in this case $G^t$ is $\gamma$-good.  If no such pair of vertices exists, then $U^{t-1}\cup V^{t-1}$ is a complete bipartite graph with, say, $|U^{t-1}|\le |V^{t-1}|$, in which case Max adds the edge $zv$.  This gives $\Delta(|U^t|)=1$ and $\Delta(|X^t|)=0$ for $X=U_1,V,V_1$.   Since $U^{t-1}\cup V^{t-1}$ is a complete bipartite graph (and since Max added no vertex to $U_1^{t-1}\cup V_1^{t-1}$), $U_1^t=V_1^t=\emptyset$, so (2*) and (4*) hold.  We have $\Delta(b_V^t)= -\gamma\le 0$, and hence $b_V^t\le 0$.  If $|U^{t-1}|< \delta= \rec{\gamma-1}$, we automatically have $b_U^t\le 0$.  Otherwise $|V^{t-1}|\ge |U^{t-1}|\ge\rec{\gamma-1}$, which implies \[|U^t|=|U^{t-1}|+1\le |V^{t-1}|+(\gamma-1)\rec{\gamma-1}\le |V^{t-1}|+(\gamma-1)|V^{t-1}|=\gamma |V^{t-1}|=\gamma|V^t|.\] Thus $b_U^t\le 0$ and (3*) holds, so $G^t$ is $\gamma$-good.
\end{proof}
For some slight ease of notation, we say that a vertex $w$ satisfies (2') if it is adjacent to some vertex in $U_0^t\cup V_0^t$.  Thus $G^t$ satisfying (2*) is equivalent to every vertex of $U^t\cup V^t$ satisfying (2').  We also note that if $w$ satisfies (2') at time $t-2$, then it will also satisfy (2') at time $t$.

\begin{claim}\label{Cl-O}
	If $e^{t-1}$ is an $O$ move, then Max can play so that $G^t$ is $\f{3}{2}$-good.
\end{claim}
\begin{proof}
	Since $b_U^{t-2}+b_V^{t-2}\le -2$, one of $b_U^{t-2}$ or $b_V^{t-2}$ must be at most $-1 $.  If $b_U^{t-2}\le -1\le -\half$, Max adds the edge $xv$, which leads to $\Delta(|U^t|)=\Delta(|V^t|)=\Delta(|V_1^t|)=1,\ \Delta(|U_1^t|)=0$.  $x$ and $y$ satisfy (2'), so (2*) continues to hold.  We have $\Delta(b_V^t)=1-\gamma\le 0$ and $\Delta(b_U^t)=2-\gamma=\half$ when $\gam=\f{3}{2}$, so $b_U^t\le 0$ since we assumed $b_U^{t-2}\le -\half$, and thus (3*) holds.  We have $\Delta(b_U^t)+\Delta(b_V^t)=3-2\gamma =0$ since $\gam=\f{3}{2}$, so (4*) holds and $G^t$ is $\f{3}{2}$-good.  If instead $b_V^{t-2}\le -1$, Max adds the edge $xu$ and essentially the same analysis gives the result. 
\end{proof}

	In response to $AU$ and $AV$ type moves, Max has to consider the overall ``State'' of $G^{t-1}$ in order to make his move.  To this end, we make the following observations.
	
	\begin{claim}\label{Cl-Overflow} ~
		
		\begin{itemize}
			\item[(a)] If $|U^{t-1}|>\gamma |V^{t-1}|+\delta$, then $e^{t-1}$ is an $AU$ move, $V_1^{t-2}=V_1^{t-1}=\emptyset$, and $b_V^{t-2}\le -1$.
			
			\item[(b)] If $|V^{t-1}|>\gamma |U^{t-1}|+\delta$, then $e^{t-1}$ is an $AV$ move, $U_1^{t-2}=U_1^{t-1}=\emptyset$, and $b_U^{t-2}\le -1$.
			
			\end{itemize}
	\end{claim}
	\begin{proof}
		For (a), assume that $|U^{t-1}|>\gamma |V^{t-1}|+\delta$.  Since we assumed that $b_U^{t-2}\le 0$, and in particular that $|U^{t-2}|\le \gamma|V^{t-2}|+\delta$ since $|V_1^{t-2}|$ is non-negative, it must be that $e^{t-1}$ is an $AU$ move, meaning $b_U^{t-1}=b_U^{t-2}+1\le 1$.  Thus $|V_1^{t-1}|=(-|U^{t-1}|+\gamma|V^{t-1}|+\delta)+b_U^{t-1}<1$, which implies that $|V_1^{t-1}|=0$ since $|V_1^{t-1}|$ is a non-negative integer, and thus $|V_1^{t-2}|=0$ as nothing is removed from $V_1^{t-2}$ by an $AU$ move.  Lastly, $b_U^{t-2}+b_V^{t-2}\le -2$ by (4*) and $b_U^{t-2}+1=b_U^{t-1}>0$, so \[b_V^{t-2}< b_V^{t-2}+b_U^{t-1}= b_V^{t-2}+b_U^{t-2}+1\le-1.\]
		This proves (a), and the analysis for (b) is similar.
	\end{proof}

	\begin{claim}\label{Cl-States}
		If $e^{t-1}$ is an $AU$ move, then the game must be in precisely one of the following three States.
		\begin{itemize}
			\item State N (Nice): $U_1^{t-1}= V_1^{t-1}=\emptyset$, $|U^{t-1}|\le \gamma|V^{t-1}|+\delta$, and $|V^{t-1}|\le \gamma|U^{t-1}|+\delta$.
			\item State OU (Overflow $U$): $|U^{t-1}|>\gamma|V^{t-1}|+\delta$ and $V_1^{t-1}=\emptyset$.  
			\item State C (Clean-Up): $|U_1^{t-1}\cup V_1^{t-1}|\ne 0$, $|U^{t-1}|\le \gamma|V^{t-1}|+\delta$, and $|V^{t-1}|\le \gamma|U^{t-1}|+\delta$.
		\end{itemize}
	\end{claim}
	\begin{proof}
		Observe that we always have $|V^{t-1}|\le \gamma|U^{t-1}|+\delta$ by Claim~\ref{Cl-Overflow} since we assume that $e^{t-1}$ is not an $AV$ move.  Assume that the game is not in State $N$. If $|U^{t-1}|>\gam |V^{t-1}|+\delta$, then Claim~\ref{Cl-Overflow} shows that $V_1^{t-1}=\emptyset$.  If $|U^{t-1}|\le \gam |V^{t-1}|+\delta$, then by assumption of the game not being in State N, we must have $|U_1^{t-1}\cup V_1^{t-1}|\ne 0$, and hence the game is in State C.
	\end{proof}

	By Claim~\ref{Cl-States}, in order to show how Max reacts to an $AU$ move, it is enough to show how he reacts to $AU$ moves that put the game into each of the States defined above.
	
	\begin{claim}
		If $e^{t-1}$ is an $AU$ move putting the game into State N, then Max can play so that $G^t$ is $\gamma$-good.
	\end{claim}
	\begin{proof}
		Max adds the edge $zu$.  With this we maintain that $U_1^t=V_1^t=\emptyset$, so (2*) and (4*) are satisfied, and we have that $\Delta(|U^t|)=\Delta(|V^t|)=1$, from which one can deduce that (3*) is satisfied.
	\end{proof}
	\begin{claim}\label{Cl-OU}
		If $e^{t-1}$ is an $AU$ move putting the game into State OU, then Max can play so that $G^t$ is $\f{3}{2}$-good.
	\end{claim}
	\begin{proof}
		Max adds the edge $zu$.  This gives $\Delta(|U^t|)=\Delta(|V^t|)=1,\ \Delta(|V_1^t|)=0,\ \Delta(|U_1^t|)\le 1$.  Clearly $z$ satisfies condition (2'), and $y$ does as well since $V_1^{t-1}=V_1^t=\emptyset$ by virtue of the game being in State OU, so (2*) is maintained.  We have $\Delta(b_U^t)=1-\gamma\le 0$ and $\Delta(b_V^t)\le 2-\gamma\le 1$, so $b_V^t\le 0$ by Claim~\ref{Cl-Overflow}, and thus (3*) is maintained.  Lastly, $\Delta(b_U^t)+\Delta(b_V^t)\le 3-2\gamma=0$ since $\gam=\f{3}{2}$, so (4*) is maintained and $G^t$ is $\f{3}{2}$-good.
	\end{proof}
	\begin{claim}\label{Cl-C}
		If $e^{t-1}$ is an $AU$ move putting the game into State C, then Max can play so that $G^t$ is $\gamma$-good.
	\end{claim}
	\begin{proof}
		First assume that $V_1^{t-1}\ne \emptyset$.  If $x\in V_1^{t-1}$, then Max adds the edge $xu$, and otherwise Max picks an arbitrary $v'\in V_1^{t-1}$ and adds the edge $v'u$.  After this we have $\Delta(|U^t|) =1,\ \Delta(|V^t|)=0,\ \Delta(|V_1^t|)=-1,\ \Delta(|U_1^t|)\le 1$. (2*) is maintained since we made sure that $y$'s neighbor $x$ was in $V_0^{t}$.  We have $\Delta(b_U^t)=0$ and $\Delta(b_V^t)\le 1-\gamma\le 0$, so (3*) is maintained.  $\Delta(b_U^t)+\Delta(b_V^t)\le 1-\gamma\le 0$, so (4*) is maintained and $G^t$ is $\gamma$-good.
		
		Now assume $V_1^{t-1}=\emptyset$, which implies $U_1^{t-1}\ne \emptyset$ since we are in State C.  In this case Max arbitrarily picks a $u'\in U_1^{t-1}$ and adds the edge $u'v$, giving $\Delta(|U^t|)=1,\ \Delta(|V^t|)=\Delta(|V_1^t|)=0,\ \Delta(|U_1^t|)\le 0$. (2*) is maintained since $y$ has a neighbor in $V^t=V_0^t$. Because we are in State C and $|V_1^t|=|V_1^{t-1}|=0$, we have \[b_U^t=|U^t|-\gam |V^t|-\delta=|U^{t-1}|-\gam |V^{t-1}|-\delta\le 0.\] We also have $\Delta(b_V^t)\le -\gamma\le 0$, so (3*) is maintained.  Lastly, $\Delta(b_U^t)+\Delta(b_V^t)\le 1-\gamma\le 0$, so (4*) is maintained and $G^t$ is $\gamma$-good.
	\end{proof}
	Thus regardless of the State, Max can react as desired to $AU$ moves, and an analogous argument works for $AV$ moves.  Thus regardless of what $e^{t-1}$ is, Max can play so that $G^t$ is $\f{3}{2}$-good whenever $G^{t-1}$ contains an isolated vertex.
\end{proof}

We emphasize for later that the only two places in this subsection where we explicitly used $\gam=\f{3}{2}$ was in the proof of Claim~\ref{Cl-O} and in verifying (4*) in the proof of Claim~\ref{Cl-OU}.  Before proving Theorem~\ref{T-5}, we require a small technical lemma.

\begin{lem}\label{L-Forever}
	For every even $t$, if $G^{t-1}$ satisfies (1*) and (2*) in the $\c{C}$-saturation game, then Max can play so that $G^t$ satisfies (1*) and (2*).
\end{lem}
\begin{proof}
	By Lemma~\ref{L-Bipartite}, $e^{t-1}$ must be an $I,\ O$, $AU$, or $AV$ type move.  If $e^{t-1}$ is of type $I$ or $O$, then Max plays exactly as he did in Proposition~\ref{P-algLow}, and one can verify that in this case (1*) and (2*) are satisfied (in particular, the strategy and proof for these properties did not utilize the isolated vertex $z$, nor properties (3*) or (4*) in these cases).  If, say, $e^{t-1}=xy$ is an $AU$ move with $x\in V^{t-2}$, then Max adds the edge $vy$, in which case one can verify that (1*) and (2*) are satisfied. A similar strategy and analysis holds for $AV$ moves.
\end{proof}

\begin{proof}[Proof of Theorem~\ref{T-5}]
	Max uses the strategy given by Proposition~\ref{P-algLow} as long as $G^{t-1}$ contains isolated vertices, after which Max uses the strategy given by Lemma~\ref{L-Forever}.   Note that this implies that $G^t$ satisfies (1*) and (2*) for all even $t$, and thus $G^t$ will be bipartite for all $t$ by Lemma~\ref{L-Bipartite}.  In particular this holds for $t=\infty$, so $e(G^\infty)=|U^\infty||V^\infty|$.  Because $|U^\infty|+|V^\infty|=n$, this product will be minimized when $||U^\infty|-|V^\infty||$ is as large as possible.  We thus need to determine how large, say, $|U^\infty|$ can be.
	
	Let $T$ denote the largest even number such that $G^{T-1}$ contains isolated vertices.  Since $G^{T+2}$ satisfies (1*) and contains no isolated vertices, $G^{T+2}$ must be connected, and hence $U^\infty=U^{T+2}$ and $V^\infty=V^{T+2}$.  Because Max followed Proposition~\ref{P-algLow} at time $T$, we have $|U^T|\le \f{3}{2} |V^T|+\delta$.  Because the only components of $G^T$ were $U^T\cup V^T$ and isolated vertices, we have $|U^{T+2}|\le |U^T|+2$.  In total we conclude that $|U^\infty|\le \f{3}{2} |V^\infty|+\delta+2$.  Thus $e(G^\infty)$ is minimized if, say, $|U^\infty|= \f{3}{2} |V^\infty|+\delta+2$, in which case we have \[n=|U^\infty|+|V^\infty|= \f{5}{2}|V^\infty|+\delta+2\implies |V^\infty|= \f{2}{5}(n-2-\delta).\] 
	We conclude that $e(G^\infty)=|U^\infty||V^\infty|\ge \f{6}{25} n^2+o(n^2)$ as desired.
\end{proof}

\section{The Lower Bound of Theorem~\ref{T-gen}}\label{S-Low2}
The strategy of Proposition~\ref{P-algLow} gives a bound for the $\c{C}_{2k+1}$-saturation game for $k\ge 2$.  When $k$ is large we can further improve upon this strategy.  Namely, Max will be able to maintain that $G^t$ is $\gam_k$-good with $\gam_k$ such that $\lim_{k\to \infty}\gam_k\to 1$, increasing the total number of edges guaranteed at the end of the game for large $k$.  The strategy will be essentially the same as before but with two key differences. First, we will identify our ``bad vertices'' as those that are at distance roughly $k$ from $u$ or $v$, instead of those that are not adjacent to $u$ or $v$.  Second, in the proof of Claim~\ref{Cl-C} we ``fixed'' a bad vertex $x$ by adding the edge, say, $xv$.  In this setting we will instead fix this vertex by adding an edge $zv$ for some ``representative'' $z$ that lies along a shortest path from $x$ to $v$.  The idea with this is that a given $z$ could represent multiple bad vertices, so adding this edge has the potential to make multiple bad vertices sufficiently close to $v$.  We note that our strategy for Max in this section is also valid for the $\c{C}$-saturation game where $\c{C}$ is any set of odd cycles with $\c{C}_{2k+1}\sub \c{C}$ and $k\ge 3$.

We now proceed with our proof of the lower bound of Theorem~\ref{T-gen}.  Throughout this section we fix some $k\ge 3$. Our notation for the proof of Proposition~\ref{P-lowAlg2} will be very similar to the notation of the proof of Proposition~\ref{P-algLow}, but we emphasize that a large portion of the notation used here will differ somewhat from how it was used before.  

Let $uv$ denote the first edge of $G^t$.  When the connected component containing $u$ and $v$ in $G^t$ is bipartite, we let $U^t\ni u$ and $V^t\ni v$ denote the parts of this bipartition.
Let $\ell=k$ if $k$ is even and $\ell=k+1$ if $k$ is odd.  Define $U_0^t=\{u'\in U^t: d^t(u',u)<\ell\}$ and $\tilde{U}_1^t=U^t\setminus U_0^t$.   Arbitrarily assign a linear ordering to the vertices of $U^t$.  We will say that a vertex $x\in U^t$ is the representative for $u'\in \tilde{U}_1^t$ if 

\begin{enumerate}
	\item $d^t(x,u)=4$.
	\item $x$ lies along a shortest path from $u'$ to $u$.
	\item $x$ is the minimal vertex (with respect to the ordering of $U^t$) satisfying these properties.
\end{enumerate}
We note that since $k\ge 3$, we have $d^t(x,u)\le d^t(u',u)$, so every $u'\in \tilde{U}_1^t$ has a representative.  Define $U_1^t$ to be the set of vertices that are representatives for some vertex of $\tilde{U}_1^t$.  Note that $|\tilde{U}_1^t|=0$ if and only if $|U_1^t|=0$.    We similarly define $V_0^t,\ \tilde{V}_1^t$, and $V_1^t$.  

For $1<\gam\le 2$ we let $\delta=\rec{\gam-1}$, and we now say that $G^t$ is $\gam$-good if it satisfies the same four conditions as we had before but with our new definitions for the sets $U_0^t,\ U_1^t,\ V_0^t$, and $V_1^t$ being used.  We first prove an analog of Lemma~\ref{L-Bipartite} using our new definitions.  Again if $G^t$ satisfies (1*) but $G^{t+1}$ has more than one non-trivial connected component, we define $U^{t+1}:=U^t$ and $V^{t+1}:=V^t$.

\begin{lem}\label{L-AltPart}
	Let $t$ be such that $G^{t}$ satisfies (1*) and (2*).  Then $U^{t+1}$ and $V^{t+1}$ are both independent sets for any valid choice of $e^{t+1}$ in the $\c{C}_{2k+1}$-saturation game for $k\ge 3$.
\end{lem}
\begin{proof}
	Let $v',v''\in V^t$, and let $u',u''\in U_0^t$ be neighbors of $v'$ and $v''$ respectively, noting that such vertices exist by (2*).  We then have
	\[d^t(v',v'')\le d^t(v',u')+d^t(u',u)+d^t(u,u'')+d^t(u'',v'')\le 1+(\ell-2)+(\ell-2)+1\le 2k,\]
	where we used that $d^t(u',u)<\ell$ is even and that $\ell\le k+1$. Since $d^t(v',v'')$ is even, having $e^{t+1}=v'v''$ would create a $C_{2k'+1}$ for some $k'\le k$, which is forbidden in the $\c{C}_{2k+1}$-saturation game. The analysis for $U^t$ is similar.
\end{proof}

\begin{prop}\label{P-lowAlg2}
	There exists a strategy for Max in the $\c{C}_{2k+1}$-saturation game when $k\ge 3$ such that for all odd $t$, whenever $G^{t-1}$ contains an isolated vertex, Max can add an edge so that $G^{t}$ is $\gam_k$-good with
	\[
		\gam_k=\f{4k^{-1}+\sqrt{16k^{-2}+4}}{2}.
	\]
\end{prop}
\begin{proof}
$G^1$ is $\gam_k$-good, so inductively assume that Max has been able to play so that $G^{t-2}$ is $\gam_k$-good with $t$ odd.  Before describing the strategy, we first make an observation about $\tilde{U}_1^{t-2}$ and $\tilde{V}_1^{t-2}$.
\begin{claim}\label{Cl-Dist}
	 $\tilde{U}_1^{t-2}=\{u':d^{t-2}(u,u')=\ell\}$ and  $\tilde{V}_1^{t-2}=\{v':d^{t-2}(v,v')=\ell\}$.
\end{claim}
\begin{proof}
	Let $u'\in U_0^{t-2}$ be a neighbor of $v'\in \tilde{V}_1^{t-2}$, which exists by (2*).  Then  \[d^{t-2}(v',v)\le d^{t-2}(v',u')+d^{t-2}(u',u)+d^{t-2}(u,v)\le 1+(\ell-2)+1=\ell,\] with the $\ell-2$ term coming from the fact that $d^{t-2}(u',u)$ is even and less than $\ell$.  Since $v'\in \tilde{V}_1^{t-2}$ implies $d^{t-2}(v',v)\ge \ell$, the distance must be exactly $\ell$.  The analysis for $\tilde{U}_1^{t-2}$ is similar.
\end{proof}

We are now ready to describe the strategy we wish to use to prove Proposition~\ref{P-lowAlg2}.  We define $I,\ O,\ AU,$ and $AV$ moves, as well as the States N, OU, and C exactly as we had written before, but we now use our new definitions for the relevant sets.  The strategy for Proposition~\ref{P-lowAlg2} is almost the same strategy as that of Proposition~\ref{P-algLow} after using our new definitions for the relevant sets, move types, and States.   The only change we make is how Max responds to an AU or AV move that brings the game into State C.  Namely, before if Mini added the edge $e^{t}=xy$ with $y$ an isolated vertex in $G^{t-1}$, we checked to see if $x$ was in, say, $U_1^t$, in which case we added the edge $xv$.  We now instead check if $x\in \tilde{U}_1^t$, and if it is we add the edge $zv$ where $z$ is the representative for $x$.  Adding this edge makes $d^t(x,u)$ strictly smaller than $d^{t-2}(x,u)$, so by Claim~\ref{Cl-Dist} we will have $d^t(x,u)<\ell$ and $y$ will have a neighbor in $U_0^t$, so (2*) will still hold.  

One can verify that with this slight modification all of the previous analysis we did in proving the claims within the proof of Proposition~\ref{P-algLow} continues to hold. It remains to address the two points in the proof of Proposition~\ref{P-algLow} where we required $\gamma=\f{3}{2}$, namely Claim~\ref{Cl-O} and Claim~\ref{Cl-OU}.

\begin{claim}\label{Cl-O'}
	If $e^{t-1}$ is an $O$ move, then Max can play so that $G^t$ is $\gam$-good.
\end{claim}
\begin{proof}
	Max reacts as he did in Claim~\ref{Cl-O}.  Observe that no vertices are added to $\tilde{U}_1^{t}$ or $\tilde{V}_1^{t}$. Indeed, the new vertices are within distance $2<\ell$ of $u$ and $v$, so they will both be added to $U_0^{t}\cup V_0^{t}$.  In particular, $\Delta(|U_1^{t}|)=\Delta(|V_1^t|)=0$, and the remaining analysis is straightforward.
\end{proof}

In order to deal with $AU$ moves putting the game into State OU, we will need the following result showing that $|U_1^t|$ is small.

\begin{claim}\label{Cl-Rep}
	For $k\ge 3$, $|U_1^t|\le 4k^{-1}|U^t|$.
\end{claim}
\begin{proof}
	For each $x\in U_1^t$, let $u_x$ denote a vertex that $x$ is the representative for, and let $P_x$ denote the set of vertices that make up a shortest path from $x$ to $u_x$.  We claim that $P_x$ and $P_y$ are disjoint if $x\ne y$.  Indeed, let $w\in P_x\cap P_y$.  If $d^t(w,x)<d^t(w,y)$, then \[d^t(u_y,w)+d^t(w,y)+d^t(y,u)=d^t(u_y,u)\le d^t(u_y,w)+d^t(w,x)+d^t(x,u).\] Since $d^t(y,u)=d^t(x,u)=4$ this implies that $d^t(w,y)\le d^t(w,x)$, a contradiction.  By using a symmetric argument we see that we must have $d^t(w,x)=d^t(w,y)$.  If we had, say, $x<y$ in the ordering of $U^t$, then $y$ could not be the representative for $u_y$ since $x$ also satisfies properties (1) and (2) for being a representative for $u_y$, meaning that $y$ is not the minimal vertex satisfying these properties.  A similar result occurs if $x>y$.  We conclude that the only way $P_x\cap P_y$ could be non-empty is if $x=y$.
	
	For each $x\in U_1^t$, we observe that the number of vertices in $P_x\cap U^t$ is \[\f{d^t(u_x,x)}{2}+1=\f{\ell-4}{2}+1=\f{\ell}{2}-1\ge \f{k}{4},\] and none of these vertices appear in any other $P_y$ for $x\ne y\in U_1^t$.  Since we can associate to each $x\in U_1^t$ a set of at least $k/4$ elements of $U^t$ without any element of $U^t$ appearing in more than one set, we must have $|U_1^t|\le 4k^{-1} |U^t|$.
\end{proof}

\begin{claim}\label{Cl-OU'}
	If $e^{t-1}$ is an $AU$ move putting the game into State OU, then Max can play so that $G^t$ is $\gam_k$-good.
\end{claim}
\begin{proof}
	Max reacts as in the proof of Claim~\ref{Cl-OU}.  As mentioned before, essentially the same proof used to prove Claim~\ref{Cl-OU} can be used here to show that $G^t$ satisfies (1*), (2*) and (3*), so it remains to verify (4*). By definition of State OU, we have $V_1^t=\emptyset$ and $|U^{t-1}|>\gam_k |V^{t-1}|+\delta$.  The latter implies that \[|V^t|=|V^{t-1}|+1< \rec{\gam_k}|U^{t-1}|-\rec{\gam_k}\delta+1=\rec{\gam_k}|U^t|-\rec{\gam_k}\delta+1.\]  Combining these observations with Claim~\ref{Cl-Rep} gives \begin{align*} b_U^t+b_V^t&=|U_1^t|+(1-\gam_k)(|U^t|+|V^t|)-2\delta \\ 
	&\le 4k^{-1} |U^t|+(1-\gam_k)\left(|U^t|+\rec{\gam_k}|U^t|-\rec{\gam_k}\delta+1\right)-2\delta\\ &=\left(-\gam_k+4k^{-1}+\rec{\gam_k}\r)|U^t|-\left(1+\rec{\gam_k}\right)\delta+(1-\gam_k)\\ 
	&=-\left(1+\rec{\gam_k}\right)\delta+(1-\gam_k),\end{align*}
	with the last equality coming from the fact that $\gam_k$ is a root of $-x^2+4k^{-1} x+1$.  It is not too hard to see that the remaining value is at most $-2$.  This shows that (4*) is maintained in this case, completing the proof.
\end{proof}
	
	To finish the proof, observe that Lemma~\ref{L-AltPart} implies that $e^{t-1}$ must be an $I,\ O,\ AU$, or $AV$ type move.  The claims we have proven here together with our work from Proposition~\ref{P-algLow} shows, provided $G^{t-1}$ contains an isolated vertex, that Max can play so that $G^t$ is $\gam_k$-good regardless of what type of move $e^{t-1}$ is, proving the statement.
\end{proof}
\begin{proof}[Proof of the lower bound of Theorem~\ref{T-gen}]
	Max follows the strategy of Proposition~\ref{P-lowAlg2} until there are no isolated vertices remaining, at which point he follows the strategy of Lemma~\ref{L-Forever}, and one can easily verify that the statement and proof of Lemma~\ref{L-Forever} continues to hold with our newly defined sets.  Using the same sort of reasoning as before, one can show $e(G^\infty)\ge\f{\gam_k}{(1+\gam_k)^2}n^2+o(n^2)$.  Note that
	\[\f{\gam_k}{(1+\gam_k)^2}=\f{k}{8}\left(\sqrt{k^2+4}-k\r)\ge \left(\quart-\f{1}{4k^2}\right),\]
	giving the desired result.
\end{proof}

\section{The Upper Bound of Theorem~\ref{T-gen}}\label{S-Up}
Throughout this section we will consider the $\c{C}_{2k-1}$-saturation game, so that the smallest odd cycle that can be made is a $C_{2k+1}$, and we will always have $k\ge 5$.  We again let $e^1=uv$.

The proof of Theorem~\ref{T-5} shows that Mini has no strategy in the $\c{C}_{2k-1}$-saturation game that guarantees the creation of any odd cycles.  On the other hand, the fact that $\sat_g(\c{C}_\infty;n)=\floor{\quart n^2}$ shows that if Mini does not attempt to make any odd cycles, then Max can force the game to end with a complete balanced bipartite graph.  Thus in order to get any non-trivial upper bound on $\sat_g(\c{C}_{2k-1};n)$, Mini must attempt to construct odd cycles, while also making sure that the game does not end as a balanced bipartite graph if she fails to construct any odd cycles.

Because of the obstructions mentioned above, the proof of the upper bound of Theorem~\ref{T-gen} needs some preparations.  The main idea of the proof is that Mini will maintain a number of long, disjoint paths in $G^t$, and then eventually either Mini will be able to join these paths together and create many odd cycles, or the graph will be ``almost'' bipartite with one side significantly larger than the other.

\subsection{Paths}\leavevmode

We wish to define a special set of disjoint paths $P^t$ in $G^t$, with each path having $v$ as one of its endpoints.  To construct this set, we first need to establish some notation. We say that $xy$ is an isolated edge if $d(x)=d(y)=1$.  Let $C_x^t$ denote the connected component containing $x$ in $G^t$.  We say that $C_x^t$ is good if $v\notin C_x^t$ and if $C_x^t$ is either an isolated vertex or an isolated edge.  If $p$ is a path in $G^t$, let $\ol{C}_p^t$ denote the connected component in $G^t-\{v\}$ containing $p-\{v\}$.   

We start with $P^1=\{uv\}$, and inductively we define $P^t$ based on the following procedure.
\begin{itemize}
	\item[Step 0.] Set $P^{t}:=P^{t-1}$.
	
	\item[Step 1.] If Mini adds the edge $e^t=xv$ with $C_x^{t-1}$ good, set $P^t:=P^t\cup \{xv\}$.
	
	\item[Step 2.] If Mini adds the edge $e^t=xw$ with $C_x^{t-1}$ good, and if there exists some $p\in P^{t-1}$ with $p=w\cdots v$, set $P^t:=(P^t\setminus\{p\})\cup \{xp\}$.
	
	
	\item[Step 3.] While there exists $p\in P^{t}$ and $x\in \ol{C}_p^t$ such that there does not exist a unique path from $x$ to $v$ in $G^t$, set $P^t:=P^t\sm\{p\}$.

\end{itemize}
We first observe that the $\ol{C}_p^t$ components are disjoint from one another.
\begin{lem}\label{L-S3}
	If $p,p'\in P^t$ and $p\ne p'$, then $\ol{C}_p^t\ne \ol{C}_{p'}^t$.
\end{lem}
\begin{proof}
	This certainly holds for $t=1$, so assume this holds up to time $t$.  Assume for contradiction that there exist $p,p'\in P^t$ such that $p\ne p$ and $\ol{C}_p^t=\ol{C}_{p'}^t$. If $p,p'\in P^{t-1}$, then we had $\ol{C}_p^{t-1}\ne \ol{C}_{p'}^{t-1}$ by our inductive hypothesis, which means we must have $e^t=xy$ with $x\in \ol{C}_p^{t-1},\ y\in \ol{C}_{p'}^{t-1}$.  But this implies that $x\in \ol{C}_p^t$ has two paths from itself to $v$, namely the path it had in $G^{t-1}$ and the path from $y$ to $v$ in $G^{t-1}$ together with the edge $xy$.  Thus $p\notin P^t$ by Step 3, a contradiction.
	
	Thus we must have, say, $p\notin P^{t-1}$.  The only way we could have $p\in P^t$ then is if $e^t=xw$ with $C_x^{t-1}$ good and either $w=v$ or $w$ the endpoint of some $p''\in P^{t-1}$.  Observe in either case that $p$ is the only new path added to $P^t$ (Step 2 can not be applied twice since our inductive hypothesis shows that there exists at most one path with $w$ as an endpoint).  Thus we must have $p'\in P^{t-1}$.  Note that $\ol{C}_{p'}^{t-1}\ne \ol{C}_{p''}^{t-1}$ by our inductive hypothesis, and that $\ol{C}_{p'}^{t-1}\ne C_x^{t-1}$ since $p'$ contains $v$ while $C_x^{t-1}$ does not since it is good.  Thus $e^t$ does not involve any vertex of $\ol{C}_{p'}^{t-1}$, and in particular $\ol{C}_{p'}^t=\ol{C}_{p'}^{t-1}\ne \ol{C}_{p''}^{t-1}\cup C_x^{t-1}\cup\{xw\}=\ol{C}_{p}^t$ as desired.
\end{proof}

Another key point with this procedure is that Step 3 does not ``interfere'' with Steps 1 and 2.
\begin{lem}\label{L-Interfere}
	For any $t$, if Step 1 or 2 adds the path $p$ to $P^t$, then $p\in P^t$.  That is, $p$ is not removed by Step 3.
\end{lem}
\begin{proof}
	If Step 1 or 2 added the path $p$, then we must have $e^t=xw$ with $C_x^{t-1}$ good.  First consider the case $w\ne v$, which means that $p=xp'$ with $p'=w\cdots v$ a path in $P^{t-1}$.  Observe that $C_x^{t-1}$ and $\ol{C}_{p'}^{t-1}$ are disjoint since $p'$ contains $v$ while $C_x^{t-1}$ does not, and hence $\ol{C}_{p}^t=\ol{C}_{p'}^{t-1}\cup C_x^{t-1}\cup \{xw\}$.
	
	Since $p'$ was not deleted by Step 3, every vertex of $\ol{C}_{p'}^{t-1}$ contains a unique path to $v$ in $G^{t-1}$.  Since no vertex of $C_x^{t-1}$ contains a path to $v$, every vertex of $\ol{C}_{p'}^t$ continues to have a unique path to $v$ in $G^t$.  Moreover, every path from $y\in C_x^{t-1}$ to $v$ in $G^t$ must consist of a path from $y$ to $x$ (possibly the empty path) followed by the unique path from $w$ to $v$.  Since $C_x^{t-1}$ is acyclic, the path from $y$ to $x$ is unique and we conclude that $p$ is not deleted by Step 3.  The proof for the case $w=v$ is similar, and we omit the details.
\end{proof}

We will be primarily interested in the endpoints of the paths of $P^t$.  To this end, let $D_\ell^t$ denote the set of vertices other than $v$ that are the endpoint of a path of length $\ell$ in $P^t$.

\begin{lem}\label{L-DL}
	Let $\ell,t\ge 1$ and assume $x\in D_\ell^t$.
	\begin{enumerate}
		\item[(a)] $d^t(x,v)=\ell$. 
		\item[(b)] $|D_\ell^{t}|-|D_\ell^{t-1}|\ge -2$, and this difference is 0 whenever Max adds an edge to $G^{t-1}$ that involves an isolated vertex of $G^{t-1}$. 
		\item[(c)] For any $t\ge 0$, if $w_1\ne w_2$ are two vertices of $D_k^t$, then choosing $e^{t+1}=w_1w_2$ is a legal move in the $\c{C}_{2k-1}$-saturation game. 
	\end{enumerate}
\end{lem}
\begin{proof}
	For (a), by definition of $D_\ell^t$ there exists a path of length $\ell$ from $x$ to $v$, and this is the only path from $x$ to $v$ by Step 3.
	
	For (b), observe that $e^t$ can involve at most two components of $G^{t-1}-\{v\}$, and each component contains at most one path of $P^{t-1}$ by Lemma~\ref{L-S3}.  Any path not in these components will not be modified or deleted by the Steps.  Hence $|D_\ell^t|-|D_\ell^{t-1}|\ge -2$.  If Max takes an edge involving an isolated vertex, then none of the Steps for modifying $P^t$ apply and we have $D_\ell^t=D_\ell^{t-1}$.
	
	For (c), let $p_i$ denote the path for which $w_i$ is an endpoint for, noting that $p_1\ne p_2$ since $w_1\ne w_2$ and neither are equal to $v$.  Since $p_1\ne p_2$, $w_1$ and $w_2$ lie in different components of $G^t-\{v\}$ by Lemma~\ref{L-S3}.  Thus if the edge $w_1w_2$ created a forbidden cycle it would have to involve the vertex $v$.  Since $d^t(w_i,v)=k$ for $i=1,2$ by part (a), the smallest cycle that could be formed is a $C_{2k+1}$, which is allowed.
\end{proof}

\subsection{Phases and Phase Transitions}\leavevmode

For the rest of the section we assume that $t$ is even.  We wish to describe each $G^t$ as belonging to a certain ``Phase'' which will determine how Mini will play.  To do this we will need some additional definitions.

Set $U^0:=\emptyset$ and $V^0:=\{v\}$.  Assume  $e^t=xy$.  If $x\in V^{t-1}$ and $y$ is an isolated vertex, then $U^t:=U^{t-1}\cup \{y\},\ V^t:=V^{t-1}$.  If $x\in V^{t-1}$ and $y$ is in an isolated edge $yz$ in $G^{t-1}$, then $U^t:=U^{t-1}\cup \{y\},\ V^t:=V^{t-1}\cup \{z\}$.  If $x\in U^{t-1}$ we define $U^t$ and $V^t$ analogously.  For any other case of $e^t$, set $U^t:=U^{t-1},\ V^t:=V^{t-1}$.  

The idea behind these definitions is that for most of the game, Mini will try and make it so that $G^t$ consists only of isolated vertices, isolated edges, and a bipartite component containing $v$.  If she achieves this for all $t\le t'$, then $U^{t'}\cup V^{t'}$ defines a bipartition for the component containing $v$.  However, at some point the graph will likely cease to have these properties, in which case $U^t\cup V^t$ will serve as an ``approximate'' bipartition.  Another feature of these sets is that they are compatible with the $D_\ell^t$ sets.

\begin{lem}\label{L-PV}
	If $x\in D_\ell^t$ for some $t$ with $\ell$ odd, then $x\in U^{t'}$ for all $t'\ge t$.  If $x\in D_\ell^t$ for some $t$ with $\ell$ even, then $x\in V^{t'}$ for all $t'\ge t$.
\end{lem}
\begin{proof}
	Note that the sets $U^t$ and $V^t$ never lose elements, so it is enough to consider the case $t'=t$ and $t$ chosen to be the minimal value such that $x\in D_\ell^t$.  First consider the case $\ell=1$.  By the Steps used to define $P^t$, having $x\in D_1^t$ and $t$ minimal implies that $e^t=xv$ with $x$ an isolated vertex or part of an isolated edge of $G^{t-1}$.  Since $v\in V^{t-1}$, this implies that $x\in U^t$, proving the statement for $\ell=1$.
	
	Now inductively assume that we have proven the statement up to $\ell>1$, and for concreteness we will assume that $\ell$ is odd.  Again by the Steps, having $x\in D_\ell^t$ implies that $e^t=xy$ for some $y\in D_{\ell-1}^{t-1}$ and that $x$ is an isolated vertex or part of an isolated edge of $G^{t-1}$. By our inductive hypothesis we have $y\in V^{t-1}$, and hence $x\in U^t$ as desired. 
\end{proof}

Let $i^t$ denote the number of isolated vertices in $G^t$ and let $c=(1000k^2)^{-1}$. We will say that $G^t$ is in Phase $\ell$ for some $-1\le \ell\le k$ based on the following set of rules.
\begin{itemize}
\item $G^0$ is in Phase 0.   

\item If $G^{t-2}$ is in Phase 0, $|D_1^t|\ge (\f{1}{9} +9c)n$, $|D_2^t|=0$, $||U^t|-|V^t||< cn$, and $\half n\le i^t\le (\half+4c)n$, then $G^t$ is in Phase 1. 

\item If $G^{t-2}$ is in Phase 1, $|D_2^t|\ge \f{1}{9} n$, $||U^t|-|V^t||< cn$,  and $i^t\ge (\f{5}{18}-9c)n$, then $G^t$ is in Phase 2.    

\item For $2\le \ell<k$, if $G^{t-2}$ is in Phase $\ell$, $|D_{\ell+1}^t|\ge(9(k-\ell-1)+4)cn$, $||U^t|-|V^t||< cn$, and $i^t\ge (8(k-\ell-1)+\sum_{j=\ell+1}^k27(k-j))cn$, then $G^t$ is in Phase $\ell+1$.   

\item If $G^{t-2}$ is in Phase $\ell$ with $\ell<k$ and $||U^t|-|V^t||\ge cn$, then $G^{t}$ is in Phase $-1$.    

\item If $G^{t-2}$ is in Phase $\ell$ and if $G^t$ satisfies none of the above situations, then $G^t$ is in Phase $\ell$.
\end{itemize}

Intuitively, these rules say that if we ever have $||U^t|-|V^t||$ large, then the game enters Phase $-1$ and never leaves it.  Otherwise the game goes from Phase $\ell$ to $\ell+1$ if $G^{t}$ contains many isolated vertices and many paths of length $\ell+1$ with $v$ as an endpoint.  We note that to leave Phase 0 we additionally require there to not be too many isolated vertices in $G^t$.  Our goal will be to show that Mini can play so that the game eventually enters either Phase $-1$ or Phase $k$, and that once the game reaches one of these Phases that Mini can play so that $e(G^\infty)$ is small.

\subsection{The Beginning of the Game}\leavevmode

We will say that a path in $U^t\cup V^t$ is alternating if the vertices  in the path alternate being in $U^t$ and $V^t$, and we define $d_a^t(x,y)$ for $x,y\in U^t\cup V^t$ to be the length of the shortest alternating path in $U^t\cup V^t$ from $x$ to $y$, with this value being infinite if no such path exists. We record some observations about this definition as a lemma.

\begin{lem}\label{L-Alt}
	Let $t,\ell\ge 1$.
	\begin{enumerate}
		\item[(a)] $d_a^t(x,v)$ is even if $x\in V^t$ and odd if $x\in U^t$.
		
		\item[(b)] If $x\in D_\ell^t$, then $d_a^t(x,v)=d^t(x,v)$.
	\end{enumerate}
\end{lem}
\begin{proof}
	The proof of (a) is immediate from our definitions.  For (b), every $x\in D_\ell^t$ has a unique path (of length $\ell$) from itself to $v$ by Step 3.  Further, this path is alternating by Lemma~\ref{L-PV}, so we conclude the result.
\end{proof} 

We now describe the kind of structure that Mini tries to preserve during the beginning of the game.  If $t$ is even and $G^t$ is in Phase $\ell$ with $0\le \ell<k$, we say that $G^t$ is $\ell$-nice if it satisfies the following three conditions.
\begin{itemize}
	\item[(1-$\ell$)] $G^t$ contains exactly one non-trivial connected component whose vertices are $U^t\cup V^t$.
	\item[(2-$\ell$)] $d_a^t(x,v)\le \ell+2$ for all $x\in U^t\cup V^t$.
	\item[(3-$\ell$)]  $i^t\ge 3$, and if $\ell\ne 0$ then $|D_\ell^t|\ge 3$.
\end{itemize}

\begin{prop}\label{P-Beg}
	Let $k\ge 5$.  Mini can play in the $\c{C}_{2k-1}$-saturation game so that, whenever $G^{t}$ is in Phase $\ell$ for some even $t\ge 0$ and $0\le \ell<k$, $G^t$ is $\ell$-nice.
\end{prop}

\begin{proof}
	For any given $t$ we say that the edge $e^{t-1}$ is of type $I$ if it involves two vertices of $U^{t-2}\cup V^{t-2}$, of type $O$ if it involves two isolated vertices of $G^{t-2}$, of  type $AU$ if it involves one isolated vertex of $G^{t-2}$ and one vertex of $V^{t-2}$, of type $AV$ if it involves one isolated vertex of $G^{t-2}$ and one vertex of $U^{t-2}$, and of type $X$ if it is not any of the four types mentioned above.  Mini's strategy is as follows, where we define $D_0^t=\{v\}$ for all $t$ to deal with the case $\ell=0$.
	
	\begin{strat}\label{SS-Strat}
		Let $t$ be such that $G^{t-2}$ is in Phase $\ell$ with $0\le \ell<k$, and assume that Max has just played $e^{t-1}$.  If $|D_\ell^{t-1}|=0$, $i^{t-1}=0$, or if $e^{t-1}$ is an $X$ move, Mini plays arbitrarily, and in this case we will say that Mini has forfeited the game.  
		
		If Mini does not forfeit, let $y\in D_\ell^{t-1}$ and let $z$ be an isolated vertex of $G^{t-1}$.  If $\ell$ is even, Mini plays as follows.
		\begin{itemize}
			\item If $e^{t-1}$ is an $I$ move, Mini plays $yz$. 
			
			\item If $e^{t-1}=xw$ is an $O$ move, Mini plays $xy$.   
			
			\item If $e^{t-1}=xu'$ is an $AV$ move with $u'\in U^{t-2}$ and $x\notin U^{t-2}\cup V^{t-2}$, Mini plays $yz$.
			
			\item Assume $e^{t-1}=xv'$ is an $AU$ move with $v'\in V^{t-2}$ and $x\notin U^{t-2}\cup V^{t-2}$. If $d_a^{t-1}(v',v)\le \ell$, Mini plays $zv$ (essentially skipping her turn).  If $d_a^{t-1}(v',v)>\ell$, let $v'x'v''\cdots v$ be a shortest alternating path from $v'$ to $v$.  Then Mini adds the edge $xv''$ if this is a legal move, otherwise she forfeits and plays arbitrarily.
		\end{itemize}
		The strategy for $\ell$ odd is exactly the same as the strategy for $\ell$ even, except that the roles of $U^t$ and $V^t$ are reversed throughout and that Mini plays $zu$ in order to ``skip her turn'' instead of $zv$.
	\end{strat}
	We note that in the $AU$ case with $\ell$ even and with $d_a^{t-1}(v',v)>\ell\ge 0$, the vertex $v''$ always exists.  Indeed, $d_a^{t-1}(v',v)$ is even by Lemma~\ref{L-Alt} and hence $d_a^{t-1}(v',v)\ge 2$, so $v''$ exists.  A similar result holds with $\ell$ odd.  With this in mind, it is not difficult to see that Mini can always follow Strategy~\ref{SS-Strat} in the $\c{C}_{2k-1}$-saturation game.  We would like to further argue that Mini never has to forfeit the game.
	
	\begin{claim}\label{Cl-F}
		If $G^{t-2}$ is in Phase $\ell$ with $0\le \ell<k$ and is $\ell$-nice, then Mini does not forfeit when using Strategy~\ref{SS-Strat} at time $t$. 
	\end{claim}
	\begin{proof}
		We only prove this when $\ell$ is even, the case when $\ell$ is odd being essentially the same. Condition (1-$\ell$) guarantees that $e^{t-1}$ is not of type $X$, and (3-$\ell$) together with Lemma~\ref{L-DL} guarantees that $|D_\ell^{t-1}|,i^{t-1}\ne 0$ (for $\ell=0$ we use that $|D_0^{t-1}|=1$ for all $t$).  Thus the only potential issue is when, say, $e^{t-1}=xv'$ is an $AU$ move with $d_a^{t-1}(v',v)>\ell$.  Assume that this is the case.
		
		Let $\tilde{G}^t=G^{t-1}\cup \{xv''\}$, and assume that $\tilde{G}^t$ contains an odd cycle $C$ of length less than $2k+1$.  Note that $C$ is not contained in $G^{t-1}$ since we assume that $G^{t-1}$ is a legal state in the $\c{C}_{2k-1}$-saturation game, so $xv''$ must be an edge of $C$.  This implies that $xv'$ is also an edge of $C$ since $x$ has degree two in $\tilde{G}^t$. If $x'$ is not a vertex of $C$, then let $C'$ be $C$ after replacing the edges $xv'$ and $xv''$ with $x'v'$ and $x'v''$.  Then $C'$ is an odd cycle of length less than $2k+1$ in $G^{t-1}$, a contradiction.  
		
		Thus $x'$ must be a vertex of $C$, which means that there exists in $G^{t-1}$ a path from $x'$ to $v'$ and a path from $x'$ to $v''$ that lie in $C$, and exactly one of these paths is of even length.  Any path of even length from $x'$ to $v'$ in $G^{t-1}$ has length at least $2k$, since otherwise this path together with the edge $x'v'$ would create a forbidden odd cycle in $G^{t-1}$.  A similar observation holds for paths of even length from $x'$ to $v''$.  We conclude that $C$ has length at least $2k+3$, which is allowed.
	\end{proof}
	With this established, we will prove our result by induction.  $G^0$ is in Phase 0 and is 0-nice.  From now on we inductively assume that $G^{t-2}$ is in Phase $\ell$ for some $0\le \ell<k$, that $G^{t-2}$ is $\ell$-nice, and further that Mini played according to Strategy~\ref{SS-Strat} for all even $t'<t$ without forfeiting, which we can assume by Claim~\ref{Cl-F}.  Assume that Mini chooses $e^t$ as prescribed by Strategy~\ref{SS-Strat}. We note that it is possible for $G^t$ to not be in Phase $\ell$.
	
	\begin{claim}\label{Cl-2'}
		$G^{t}$ satisfies (1-$\ell$) and (2-$\ell$).
	\end{claim}
	\begin{proof}
		It is not difficult to see that (1-$\ell$) is maintained.  In verifying (2-$\ell$), we only consider the case $\ell$ even, the analysis for the odd case being exactly the same but with the roles of $U^t$ and $V^t$ reversed throughout.  Observe that in this case we have $y\in V^{t-1}$ by Lemma~\ref{L-PV} and that $d_a^t(y,v)=\ell$ by Lemma~\ref{L-Alt}.
		
		If $e^{t-1}$ is an $I$ move, then $d_a^t(z,v)=\ell+1$ since $d_a^t(y,v)=\ell$, and no other distances increase, so (2-$\ell$) is maintained.  
		
		If $e^{t-1}$ is an $O$ move, observe that $xw$ is an isolated edge in $G^{t-1}$.  Because of this and the fact that $y\in V^{t-1}$, we have $x\in U^t$ and $w\in V^t$ by how these sets are defined.   We then have $d_a^t(x,v)=\ell+1$ and $d_a^t(w,v)=\ell+2$, so (2-$\ell$) is maintained. 
		
		If $e^t$ is an $AV$ move, note that $d_a^{t-1}(u',v)$ is odd by Lemma~\ref{L-Alt}, and hence at most $\ell+1$ by (2-$\ell$) and the fact that $\ell+2$ is even.  Thus $d_a^t(x,v)\le \ell+2$, and we also have $d_a^t(z,v)=\ell+1$, so (2-$\ell$) is maintained. 
		
		If $e^t$ is an $AU$ move with $d_a^{t-1}(v',v)\le \ell$ then it is not hard to see that (2-$\ell$) is maintained.  Otherwise we have $d_a^t(x,v)=\ell+1$ and (2-$\ell$) is maintained.
	\end{proof}
	Let $t_\ell$ denote the smallest even value such that $G^{t_\ell}$ is in Phase $\ell$.  Observe that the game can not leave Phase $\ell$ and come back to it at a later time, which means that $G^{t'}$ is in Phase $\ell$ for all even $t'$ with $t_\ell\le t'<t$.  For any even $t'$ with $t_\ell\le t'\le t$, define $g^{t'}=|V^{t'}|-|U^{t'}|$ if $\ell$ is even and $g^{t'}=|U^{t'}|-|V^{t'}|$ if $\ell$ is odd.  Recall the definition $\Delta(X^{t'}):=X^{t'}-X^{t'-2}$ for any relevant $X$, and that we assumed that Mini used Strategy~\ref{SS-Strat} without forfeiting for all even $t'$ with $t_\ell< t'\le t$.
	\begin{claim}\label{Cl-Change}
		Let $t'$ be even with $t_\ell< t'\le t$.  If $\ell>0$ is even, then the following hold.
		\begin{itemize}
			\item  If $e^{t'-1}$ is of type $I$: $\Delta(|D_{\ell+1}^{t'}|)\ge -1,\ \Delta(|D_{\ell}^{t'}|)\ge -3,\ \Delta(g^{t'})=-1,\ \Delta(i^{t'})=-1$.
			
			\item If $e^{t'-1}$ is of type $O$: $\Delta(|D_{\ell+1}^{t'}|)=1,\ \Delta(|D_{\ell}^{t'}|)=-1,\ \Delta(g^{t'})=0$, $\Delta(i^{t'})=-2$.
			
			\item If $e^{t'-1}$ is of type $AV$: $\Delta(|D_{\ell+1}^{t'}|)= 1,\ \Delta(|D_{\ell}^{t'}|)= -1$, $\Delta(g^{t'})= 0$, $\Delta(i^{t'})=-2$.
			
			\item If $e^{t'-1}$ is of type $AU$: $\Delta(|D_{\ell+1}^{t'}|)\ge 0,\ \Delta(|D_{\ell}^{t'}|)=0,\ \Delta(g^{t'})\le -1,\ -1\ge \Delta(i^{t'})\ge -2$.
		\end{itemize}
		The same results hold for $\ell=0$ when one ignores the $\Delta(|D_{\ell}^{t'}|)$ terms.  Analogous results hold for $\ell$ odd by switching the results for $AU$ with those of $AV$.
	\end{claim}
	\begin{proof}
		This is not particularly difficult to verify. In particular one uses Lemma~\ref{L-Interfere} to show that $|D_{\ell+1}^{t'-1}|$ increases after Mini responds to $I,\ O$, and $AU$ moves, and Lemma~\ref{L-DL} to bound the changes in $|D_{\ell}^{t'-2}|$ and $|D_{\ell+1}^{t'-2}|$ after Max plays.  We omit the details.
	\end{proof}
	In order to show that $G^t$ satisfies (3-$\ell$), we will use Claim~\ref{Cl-Change} together with the following claim which shows that $|D_\ell^{t_\ell}|$ and $i^{t_\ell}$ are large.
	\begin{claim}\label{Cl-Init}
		We have the following.
		\begin{itemize}
			\item[(a)] For $\ell\ge 2$ we have $|D_\ell^{t_\ell}|\ge (9(k-\ell)+4)cn$ and $i^{t_\ell}\ge (8(k-\ell)+\sum_{j=\ell}^k 27(k-j))cn$.  
			\item[(b)] For any $\ell$ with $0\le \ell<k$, $G^{t_\ell}$ satisfies (3-$\ell$) when $n$ is sufficiently large.
		\end{itemize}
		
	\end{claim}
	\begin{proof}
		Part (a) is immediate for $\ell>2$ by how we defined our Phases.  To deal with the case $\ell=2$, recall that $c=(1000k^2)^{-1}$.  We have
		\[
			|D_{2}^{t_2}|\ge \rec{9} n\ge \f{9}{1000k} n\ge \f{9k-14}{1000k^2}n=(9(k-2)+4)cn,
		\]
		\[
			i^{t_2}\ge \left(\f{5}{18}-9c\r)n \ge \f{35}{1000}n\ge \f{8k+27k^2}{1000k^2}n\ge (8(k-2)+\sum_{j=2}^k 27(k-j))cn. 
		\]
		
		Part (b) for $\ell\ge 2$ follows from (a).  The statement is true for $\ell=0$ when $n\ge 3$ and is true for $\ell=1$ when $n\ge 27$ since $|D_{1}^{t_1}|,i^{t_1}\ge \rec{9}n$ by the rule for moving from Phase 0 to Phase 1.
	\end{proof}

	\begin{claim}\label{Cl-Turn}
		Let $\ell'$ denote the Phase of $G^t$. If $n$ is sufficiently large, then either $\ell'=-1$, $\ell'=k$, or $G^t$ satisfies (3-$\ell'$).
	\end{claim}
	
	\begin{proof}
		Note that by the way our Phases were defined, we must have $\ell'=-1,\ \ell$, or $\ell+1$.  There is nothing to prove if $\ell'=-1$, so we can assume that $\ell'\ne -1$.  Consider the case that $\ell'=\ell+1$.  If $\ell=k-1$ then there is nothing to prove, and otherwise we have $t=t_{\ell'}$ since $G^{t-2}$ was in Phase $\ell$, which means that $G^t$ satisfies (3-$\ell'$) by Claim~\ref{Cl-Init}.   We can thus assume that $\ell'=\ell$.
		
		Assume first that $\ell$ is even.  For any even $t'$ with $t_\ell<t'\le t$, let $r_1^{t'}$ denote the number of even $t''$ with $t_\ell <t''\le t'$ such that $e^{t''-1}$ was of type $I$ or $AU$, and similarly define $r_2^{t'}$ to correspond to the number of $O$ and $AV$ moves.  Observe that we always have $r_1^{t'}+r_2^{t'}=\half (t'-t_\ell)$ since we assume that Mini has used Strategy~\ref{SS-Strat} up to time $t'$,  and hence Max never played an $X$ move by Claim~\ref{Cl-F}. By Claim~\ref{Cl-Change} we have that $g^{t'}\le g^{t_\ell} -r_1^{t'}$.  Note that $g^{t_\ell}< cn$ and $g^{t'}>-cn$, as otherwise we would have either $G^{t_\ell}$ or $G^{t'}$ in Phase $-1$.  Thus we can assume that $r_1^{t'}\le 2cn$, and hence $r_2^{t'}=\half (t'-t_\ell)-r_1^{t'}\ge \half (t'-t_\ell)-2cn$.  By using this, Claim~\ref{Cl-Change}, and $r_1^{t'}+r_2^{t'}=\half (t'-t_\ell)$, we conclude that
		\begin{equation}\label{E-D1}
			|D_\ell^{t'}|\ge |D_{\ell}^{t_\ell}|-3r_1^{t'}-r_2^{t'}\ge |D_{\ell}^{t_\ell}|-\half (t'-t_\ell)-4cn,
		\end{equation}
		\begin{equation}\label{E-D2}
			|D_{\ell+1}^{t'}|\ge |D_{\ell+1}^{t_\ell}|-r_1^{t'}+r_2^{t'}\ge |D_{\ell+1}^{t_\ell}|+\half (t'-t_\ell)-4cn,
		\end{equation}
		\begin{equation}\label{E-I}
			i^{t'}\ge i^{t_\ell}-2r_1^{t'}-2r_2^{t'}=i^{t_\ell}-(t'-t_\ell),
		\end{equation}
		\begin{equation}\label{E-I2}
			i^{t'}\le i^{t_\ell}-r_1^{t'}-2r_2^{t'}\le i^{t_\ell}-(t'-t_\ell)+4cn.
		\end{equation}
		
		First consider the case $\ell=0$, which implies that $t_0=0,\ |D_1^{t_0}|=0$, and $i^{t_0}=n$.  Observe that using Strategy~\ref{SS-Strat} we have $|D_2^{t'}|=0$ for any $t'$ with $G^{t'}$ in Phase 0.  Let $\tilde{t}$ be an even integer such that $(\f{2}{9}+26c)n\le \tilde{t}\le \half n$, which exists when $n$ is sufficiently large since $\half -(\f{2}{9}+26c)>0$.  If $t'$ is even with $(\f{2}{9}+26c)n\le t'\le \min(t,\tilde{t})\le \half n$, then Equations~\eqref{E-D2}, \eqref{E-I}, and \eqref{E-I2} imply that $|D_1^{t'}|\ge (\rec{9}+9c)n$ and $\half n\le i^{t'}\le (\half n+4c)n$.  This implies that $G^{t'}$ is in Phase 1 since $G^{t'-2}$ was in Phase 0.  But $G^{t'}$ was in Phase 0 since $t'\le t$ and we assumed $G^t$ was in Phase 0, a contradiction.  Thus no such $t'$ exists, and in particular we must have $t<(\f{2}{9}+26c)n\le \half n$.  Equation~\eqref{E-I} then implies that $i^t\ge \half n\ge 3$ for $n\ge 6$, so $G^t$ satisfies (3-0).
		
		Now assume $\ell>0$ is even. Let $\tilde{t}$ be the smallest even integer such that $\tilde{t}\ge (16+18(k-\ell-1))cn+t_\ell$.  If $\tilde{t}\le t$, then Claim~\ref{Cl-Init} together with Equations~\eqref{E-D2} and \eqref{E-I} imply that
		\[
			|D_{\ell+1}^{\tilde{t}}|\ge \half (16+18(k-\ell-1))cn-4cn= (4+9(k-\ell-1))cn,
		\]
		and
		\begin{align*}
		i^{\tilde{t}}&\ge (8(k-\ell)+\sum_{j=\ell}^k 27(k-j))cn-(16+18(k-\ell-1))cn-2\\ &= (8(k-\ell-1)+\sum_{j=\ell+1}^k 27(k-j)+9(k-\ell)-8)cn-2\\ &\ge (8(k-\ell-1)+\sum_{j=\ell+1}^k 27(k-j))cn+3,
		\end{align*}
		for $n$ such that $cn\ge 5$.  Thus $G^{\tilde{t}}$ is in Phase $\ell+1$, a contradiction, so $t\le (16+18(k-\ell-1))cn+t_\ell$.  Assuming this, one can go through the same calculations as above and conclude that $i^t\ge 3$, and we also have
		\[
			|D_\ell^t|\ge (9(k-\ell)+4)cn-\half \cdot (16+18(k-\ell-1))cn-4cn=cn\ge 3
		\]
		when $n$ is sufficiently large.  
		
		The analysis for $\ell\ge 3$ odd is essentially the same as above after switching the roles of $AU$ and $AV$ when defining $r_1^{t'}$ and $r_2^{t'}$.  The analysis is almost the same for $\ell=1$, except we use $|D_1^{t_1}|\ge \f{1}{9} n+9cn$, $i^{t_1}\ge \half n$, and $\tilde{t}\ge (\f{2}{9}+8c)n+t_1$.
	\end{proof}
	
	By Claim~\ref{Cl-Turn} and Claim~\ref{Cl-2'}, if $G^t$ is in Phase $\ell'$ with $\ell'\ne -1,k$, then $G^t$  satisfies (1-$\ell$), (2-$\ell$), and (3-$\ell'$).  Since (1-$\ell$) and (2-$\ell$) imply (1-$\ell'$) and (2-$\ell'$) when we have $\ell'=\ell$ or $\ell'=\ell+1$, $G^t$ is $\ell'$-nice and we conclude the result by induction.
	
\end{proof}
\subsection{Endgame}\leavevmode
It remains to describe Mini's strategy in Phases $-1$ and $k$, and to argue that with this strategy $G^\infty$ will end up with few edges.

\begin{prop}\label{P-k}
	Assume $G^{t_k}$ is in Phase $k$ in the $\c{C}_{2k-1}$-saturation game with $k\ge 5$ and such that $t_k$  is the minimum value for which this holds.  Then Mini can play so that $e(G^{\infty})\le \quart(1-c^2)n^2+o(n^2)$.
\end{prop}
\begin{proof}
	Mini's strategy is as follows.  Let $t\ge t_k$ be even. If $|D_k^{t-1}|\ge 2$, Mini plays $e^t=xy$ with $x,y\in D_k^{t-1}$, which is a legal move by Lemma~\ref{L-DL}.  Otherwise Mini plays arbitrarily.
	
	Note that $\Delta(|D_k^t|)\ge -4$ by Lemma~\ref{L-DL}.  Since $|D_k^{t_k}|\ge 4cn$, $G^\infty$ will contain at least $cn-1$ $C_{2k+1}$'s which all share a common vertex $v$ and with no other vertices shared between the cycles. Let $\c{C}$ denote a set of $cn-1$ such cycles.  We wish to show that there are few edges involving vertices of $\c{C}$.
	
	\begin{claim}\label{Cl-cycneigh}
		If $C$ is a $C_{2k+1}$ in $G^\infty$, then every vertex of $G^\infty$ has at most two neighbors in $C$.
	\end{claim}
	\begin{proof}
		Let $v'$ be a vertex with neighbors $v_1,v_2,v_3\in C$. First assume $v'\in C$, and without loss of generality that $v_1v',v_2v'$ are edges in $C$.  Thus there exist paths from $v'$ to $v_3$ of lengths $\ell$ and $2k+1-\ell$ for some $\ell\ge 2$. One of these path lengths must be even and at most $2k-1$.  By adding the edge $v_3v'$ to this path, we find an odd cycle in $G^\infty$ of length at most $2k$, a contradiction.
		
		Now assume $v'\notin C$.  Let $P_{12}$ denote the path in $C$ from $v_1$ to $v_2$ that does not contain $v_3$, and similarly define $P_{23}$ and $P_{13}$.   Note that $e(P_{12})+e(P_{23})+e(P_{13})=2k+1$, so at least one of these paths must be of odd length.  If $e(P_{ij})=1$ for any $i,j$, then $\{v_i,v_j,v'\}$ defines a $C_3$ in $G^\infty$, which is forbidden, so assume that this is not the case.  We conclude that for some $i,j$ we have $e(P_{ij})\le 2k-3$ odd, which together with $v'$ implies the existence of an odd cycle of length at most $2k-1$ in $G^\infty$, a contradiction.
	\end{proof}
	By Claim~\ref{Cl-cycneigh}, each cycle of $\c{C}$ is involved in at most $2n$ edges, so the number of edges involving some vertex of $\c{C}$ is at most $2n(cn-1)$. By Mantel's theorem, the number of edges involving vertices that are not in $\c{C}$ is at most 
	
	\[\rec{4}(n-2k(cn-1)-1)^2=\left(\quart-kc+k^2c^2\r)n^2+o(n^2).\]
	
	In total then the number of edges in $G^\infty$ is at most
	\[
		\left(\quart -(k-2)c+k^2c^2\r)n^2+o(n^2).
	\]
	One can verify that $(k-2)c-k^2c^2\ge c^2/4$ for $k\ge5$, from which the result follows.
\end{proof}

We now deal with Phase $-1$.

\begin{prop}\label{P-1}
	Assume that $G^{\tilde{t}}$ is in Phase $-1$ in the $\c{C}_{2k-1}$-saturation game with $k\ge 5$ and $\tilde{t}$ the minimum value for which this holds.  Then Mini can play so that $e(G^{\infty})\le \quart(1-c^2)n^2+o(n^2)$.
\end{prop}
\begin{proof}
	Let $\ell<k$ be the number such that $G^{\tilde{t}-2}$ was in Phase $\ell$. Mini's strategy is as follows.  If $G^{t-1}$ is connected, Mini plays arbitrarily.  Otherwise Mini plays almost the same way as in Strategy~\ref{SS-Strat} with parameter $\ell$, with the only changes being that Mini does not forfeit if $|D_\ell^t|=0$, if $\ell$ is even we replace anywhere $y$ is written in Strategy~\ref{SS-Strat} with $v$, and if $\ell$ is odd we replace $y$ with $u$.  If one goes back through the analysis of Proposition~\ref{P-Beg}, one can verify that with this strategy, for all even $t\ge \tilde{t}$, $G^t$ satisfies (1-$\ell$), (2-$\ell$), and that  $||U^{t}|-|V^t||\ge cn-1$.
	
	Let $U'=\{u'\in U^\infty: \exists u''\in U^\infty,\ u'u''\in E(G^\infty)\}$ and $V'=\{v'\in V^\infty:\exists v''\in V^\infty,\ v'v''\in E(G^\infty)\}$.  Let $D'_\ell$ denote the set of vertices in $G^\infty$ that were in $D_\ell^t$ for some $t$. 
	
	\begin{claim}\label{Cl-Adj}
		No vertex $w\in U'\cup V'$ has $d_a^\infty(w,v)<k-1$.  No vertex of $U'$ is adjacent to any vertex of $D_2'$, and no vertex of $V'$ is adjacent to any vertex of $D_1'$. 
	\end{claim}
	\begin{proof}
		Let $u_1,u_2\in U'$ be such that $u_1u_2$ is an edge in $G^\infty$, and assume $d_a^\infty(u_1,v)<k-1$.  Let $p_i$ denote a shortest alternating path from $u_i$ to $v$, and let $w$ be the vertex that is in both $p_1$ and $p_2$ and with $d_a^{\infty}(w,v)$ maximal (such a $w$ exists since in particular $v$ is in both of these paths).  Observe that the parity of $d_a^\infty(u_i,w)$ is independent of $i$ and that $d_a^\infty(u_2,w)\le d_a^\infty(u_2,v)\le k+1$ since $G^\infty$ satisfies (2-$\ell$).  Thus $G^\infty$ contains an odd cycle of length $d_a^\infty(u_1,w)+d_a^\infty(w,u_2)+d(u_1,u_2)<2k+1$, a contradiction.
		
		If $u'v'$ were an edge in $G^\infty$ for some $u'\in U'$ and $v'\in D_2'$, then we would have $d_a^\infty(u_1,v)\le 3<k-1$ when $k\ge 5$, a contradiction to what we have just proven. The proof for $V'$ is analogous.
	\end{proof} 
	We note that the only place we truly use the hypothesis $k\ge 5$ in this section is in proving the second part of the above claim.

	With the above claim in hand, assume first that $G^\infty$ is not bipartite, or equivalently that $U'\cup V'$ is non-empty since $G^\infty$ satisfies (1-$\ell$).  If $t$ is such that $G^t$ was in Phase 1, then Proposition~\ref{P-Beg} shows that any $x$ with $d_a^t(x,v)>3$ must be isolated in $G^t$.  In particular, since $w\in U'\cup V'$ has $d_a^t(w,v)\ge d_a^\infty(w,v)\ge k-1\ge 4$ by Claim~\ref{Cl-Adj}, every such $w$ was isolated during all of Phase 1.  Further, $d_a^\infty(w,v)\ge 4$ implies that $U'\cup V'$ will be empty unless $\ell\ge 2$ since Mini maintains (2-$\ell$).  In particular, there exists a smallest even number $t_2$ such that $G^{t_2}$ is in Phase 2 with $|D_2^{t_2}|\ge \f{1}{9} n$.  Observe then that $G^{t_2-2}$ is in Phase 1, $|D_2^{t_2-2}|\ge \f{1}{9}n-1$, none of the vertices of $D_2^{t_2-2}$ are isolated at time $t_2-2$, and all of these vertices were isolated at the beginning of Phase 1 since we require $|D_2^t|=0$ in order to transition to Phase 1. Since Phase 1 starts with at most $(\half+4c)n$ isolated vertices, we conclude that $s:=|U'\cup V'|\le (\half +4c)n-(\f{1}{9} n-1)$.
	
	Let $G'$ be the complete bipartite graph with bipartition $U^\infty\cup V^\infty$, where we note that we have $G'=G^\infty$ if $s=0$.  The only edges of $G^\infty$ that are not in $G'$ are those contained in $U'\cup V'$, and there are at most $\quart s^2+1$ such edges by Mantel's theorem.  However, $G'$ contains all of the edges from $D_2'$ to $U'$ and $D_1'$ to $V'$, and none of these edges are in $G^\infty$ by Claim~\ref{Cl-Adj}.  There are at least $|D_2'||U'|+|D_1'||V'|\ge \f{1}{9} ns$ edges of this kind, so in total $G'$ contains at least $\f{1}{9} ns-\quart s^2-1$ more edges than $G^\infty$ does. One can verify that this number is non-negative if $s\ne 0$ and if $n$ is sufficiently large by our bound on $s$ and how we defined $c$. Thus it is enough to give an upper bound for $e(G')=|U^\infty||V^\infty|$.  Since $||U^\infty|-|V^\infty||\ge cn-1$, we have \[|U^\infty||V^\infty|\le \left(\half n-\half(cn-1)\r)\left(\half n+\half(cn-1)\r)=\quart \left(1-c^2\r)n^2+o(n^2).\]
\end{proof}

We are now ready to finish our proof of Theorem~\ref{T-gen}.

\begin{proof}[Proof of the upper bound of Theorem~\ref{T-gen}]
	Recall that Theorem~\ref{T-gen} is stated in terms of the $\c{C}_{2k+1}$-saturation game as opposed to the $\c{C}_{2k-1}$-saturation game.  In particular, $c=(1000(k+1)^2)^{-1}$ and $G^t$ can be in Phase $\ell$ for any $-1\le \ell\le k+1$. In the $\c{C}_{2k+1}$-saturation game, Mini plays as in Proposition~\ref{P-Beg} as long as $G^t$ is not in Phase $-1$ or $k+1$.  If the game ever enters Phase $-1$ or $k+1$, then she plays as in Proposition~\ref{P-1} or Proposition~\ref{P-k}, respectively.
	
	By Proposition~\ref{P-Beg}, $i^t\ge 3$ whenever $G^t$ is not in Phase $-1$ or $k+1$.  Since $i^\infty=0$, we must have $G^\infty$ in Phase $-1$ or $k+1$, and in particular Mini must have played according to either Proposition~\ref{P-1} or Proposition~\ref{P-k}.  By these propositions, $G^\infty$ contains at most $(\quart-\quart c^2)n^2+o(n^2)$ edges. Plugging in $c$ and using $k+1\le 2k$ for $k\ge 4$ gives the result.
\end{proof}

\section{Proof of Theorem~\ref{T-OneCyc}}
In order to prove Theorem~\ref{T-OneCyc} we need to argue that Mini can create a certain subgraph in $G^t$.
\begin{lem}\label{L-Sub}
	Let $k\ge2$ and $\ell=\max(3,\floor{\sqrt{2k}})$.  There exists a constant $t_0$ such that, for $n$ sufficiently large, Mini can play in the $\{C_{2k+1}\}$-saturation game such that $G^{t_0}$ contains a clique on the vertex set $U=\{u_1,\ldots,u_\ell\}$, and such that there exist $\ell$ vertex disjoint paths of length $k-2$, each with a distinct $u_i$ as its endpoint.
\end{lem}
\begin{proof}
	Mini will use the following strategy.
	
	\begin{strat}\label{SS-Strat2}
		~
		\begin{itemize}
			\item[Step 0.] If there exists some $u_i,u_j$ such that $G^{t-1}$ does not contain the edge $u_iu_j$, Mini plays $e^t=u_iu_j$.
			
			\item[Step 1.] Let $t'$ be the smallest even value such that $G^{t'-1}$ contains every edge of the form $u_iu_j$.  Mini plays $e^{t'}=u_1x_1,\ e^{t'+2}=x_1x_2,\ldots,\ e^{t'+2k-6}= x_{k-3}x_{k-2}$, where $x_i$ is some isolated vertex in $G^{t'+2i-3}$.
			
			\item[Step 2.] Mini plays $e^{t'+2k-4}=u_2y_1,\ e^{t'+2k-2}=y_1y_2,\ldots,\ e^{t'+4k-12}=y_{k-3}y_{k-2}$, where $y_i$ is some isolated vertex in $G^{t'+2k+2i-7}$.

		\end{itemize}
		One defines Step $i$ for all $3\le i\le \ell$ in essentially the same way as Steps 1 and 2.
	\end{strat}
	Observe that if Mini can use the above strategy, then she finishes at time at most $t_0:=2{\ell \choose 2}+2\ell(k-2)$ with the desired structure.  Thus its enough to argue that she can indeed use this strategy when $n$ is sufficiently large. Since $t_0$ is a constant, for $n$ sufficiently large there will always exist an isolated vertex in $G^{t-1}$ for $t\le t_0$, so Mini can play as prescribed by Steps 1 through $\ell$ if the game reaches this point.  It remains to argue that Mini can plays as prescribed by Step 0.
	
	If $\ell=3$, then it is not too difficult to see that Mini can play as prescribed by Step 0 regardless of what Max does, so assume $\ell=\sqrt{2k}\ge 3$. For any $t\le 2{\ell \choose 2}$, we claim that any choice of $e^t$ is a legal move. Indeed, for any such $t$, $G^{t-1}$ will contain at most $2{\ell \choose 2}-1\le \ell^2-1\le 2k-1$ edges, and hence any choice of $e^{t}$ will not create a $C_{2k+1}$ in $G^{t}$.  Thus Mini can play according to Strategy~\ref{SS-Strat2} and we conclude the result.
\end{proof}

\begin{proof}[Proof of Theorem~\ref{T-OneCyc}]
	Mini first uses the strategy in Lemma~\ref{L-Sub}, making sure that $G^{t_0}$ contains a clique on $U=\{u_1,\ldots,u_\ell\}$ and vertex disjoint paths $\{p_1,\ldots,p_\ell\}$, each of length $k-2$ with $p_j$ starting at $u_j$ and ending at, say, $v_j$.  Let $V=\{v_1,\ldots,v_\ell\}$, and let $v^t$ denote a $v_j$ with minimal degree in $G^t$. Let $i^t$ denote the number of isolated vertices of $G^t$.
	
	For all even $t>t_0$, Mini uses the following strategy.  If $i^{t-1}=0$, Mini plays arbitrarily.  Otherwise if Max plays $xy$ with $x,y$ isolated vertices of $G^{t-1}$, Mini plays $xv^t$ (which is a legal move since this does not create a cycle).  Otherwise Mini plays $xv^t$ with $x$ an isolated vertex of $G^{t-1}$.  
	
	We wish to bound the number of edges of $G^\infty$ when Mini uses the above strategy.  To this end, let $P$ denote the vertices that belong to some $p_j$ (including $u_j$ and $v_j$), let $V_j=N^\infty(v_j)$, let $V=\bigcup V_j$, and let $W=V(G^\infty)\setminus (P\cup V)$.  Let $p'_j$ denote $p_j$ but with $p'_j$ treated as a path from $v_j$ to $u_j$.  Lastly, for $X,Y\sub V(G^\infty)$, let $e(X,Y)$ denote the number of edges in $G^\infty$ where one vertex lies in $X$ and the other in $Y$.
	
	\begin{claim}\label{Cl-Edges}
		The following bounds hold.
		\begin{itemize}
			\item[(a)] $e(P,V(G^\infty))\le \ell(k-1)n=o(n^2)$. 
			\item[(b)] $e(V,V)\le k\cdot \f{2k-1}{2}n=o(n^2)$.
			\item[(c)] $e(W,V)\le (\half(1+\f{1}{\ell})n-|W|)|W|+o(n^2)$.
		\end{itemize}
	\end{claim}
	\begin{proof}
		(a) follows from $|P|= \ell(k-1)$.
		
		For (b), we first claim that $e(V_j,V_j)\le \f{2k-1}{2}n$. Indeed if this were not the case, then by the Erd\H{o}s-Gallai Theorem there would exist a path of length $2k$ in $V_j$.  Since $v_j$ is adjacent to the two endpoints of this path, this would imply that $G^\infty$ contains a $C_{2k+1}$, a contradiction.  We also claim that $e(V_j,V_{j'})=0$ whenever $j\ne j'$.  Indeed assume that $G^{\infty}$ contained the edge $w_jw_{j'}$ with $w_j\in V_j,\ w_{j'}\in V_{j'}$. Then for any $r\ne j,j'$ (and such an $r$ exists since $\ell\ge 3$), $G^\infty$ would contain the cycle $w_jp_j' u_{r} p_{j'} w_{j'}$.  But this is a $C_{2k+1}$, a contradiction.  We conclude that (b) holds.
		
		For (c), we claim that, for any $w\in W$, we have $e(\{w\},V_j)\ne 0$ for at most one $j$. Indeed assume $G^{\infty}$ contained the edges $ww_j$ and $ww_{j'}$ with $w_j\in V_j,\ w_{j'}\in V_{j'},\ j\ne j'$.  Then $G^\infty$ would contain the cycle $ww_jp_j'p_{j'}w_{j'}$, which is a $C_{2k+1}$, a contradiction.  It follows that $e(W,V)\le |W|\max(|V_j|)$, so it will be enough to bound $\max(|V_j|)$.
		
		Note that $i^{t_0}\ge n-2t_0$ and $\Delta(i^{t})\ge -2$ for all even $t\ge t_0+2$ by the way the strategy was constructed. It follows that there are at least $n/2+O(1)$ even values of $t$ with $i^{t-1}\ne 0$, and hence Mini adds an edge of the form $xv^t$ for at least this many values of $t$.  Thus Mini ensures that each of the $\ell$ vertices $v_j$ have at least $ \f{n}{2\ell}+O(1)$ neighbors in $G^\infty$, and hence $|V_j|\ge \f{n}{2\ell}+O(1)$ for all $i$.  Thus \[|V_j|=n-\sum_{j'\ne i} |V_{j'}|-|W|-|P|\le \left(1-\f{\ell-1}{2\ell}\r)n-|W|+O(1)\]
		for all $j$.  Plugging this bound into $e(W,\bigcup V_j)\le |W|\max(|V_j|)$ and using  $(1-\f{\ell-1}{2\ell})=\half(1+\rec{\ell})$  shows that (c) holds.
	\end{proof}
	
	By Claim~\ref{Cl-Edges} and Mantel's theorem, we have
	\begin{align*}
	e(G^\infty)&\le e(W,W)+e(W,V)+o(n^2)\\ &\le \quart |W|^2+\left(\half(1+\f{1}{\ell})n-|W|\right)|W|+o(n^2).
	\end{align*}
	This value is maximized when $|W|=\f{1}{3}(1+\f{1}{\ell})n$, giving an upper bound of $\f{1}{12}(1+\rec{\ell})^2n^2+o(n^2)$
	as desired.
\end{proof}

\section{Proof of Theorem~\ref{T-Mod3}}
We will say that a vertex $v$ is good if all but at most one edge incident to $v$ is contained in a triangle.  We will say that a graph $G$ is $k$-good if there exists a set of edges $B(G)$ with $|B(G)|\le k$ such that every vertex of $G-B(G)$ is good.  Observe that if $G$ is $k$-good and $G'$ is $G$ plus an edge, then $G'$ is $(k+1)$-good.

\begin{prop}\label{P-3Alg}
	There exists a strategy for Mini in the $(\c{C}_{\infty}\sm \{C_3\})$-saturation game such that for all even $t$, either $G^{t-1}$ is $(\c{C}_{\infty}\sm \{C_3\})$-saturated or $G^t$ is 1-good.
\end{prop}

To prove this, we will need the following lemma concerning the structure of 2-good graphs.
\begin{lem}\label{L-Good}
	Let $G$ be a 2-good graph that contains no $C_{2k+1}$ for any $k\ge 2$.  Then $G$ contains no $C_{\ell}$ for any $\ell \ge 5$.
\end{lem}
\begin{proof}
	Assume for contradiction that there exists an even cycle $C$ in $G$ of length $2k$ with $k\ge 3$ on the vertex set $\{v_1,\ldots,v_{2k}\}$, and let $C'=C- B(G)$.  Since $k\ge 3$, there exists an $i$ such that $C'$ contains the edges $v_{i-1}v_i$ and $v_iv_{i+1}$.  Since these edges are in $G-B(G)$, at least one of these edges is in a triangle, say $v_iv_{i+1}w$ is a triangle in $G$.  If $w$ is not in $C$, then $v_1v_2\cdots v_i wv_{i+1}\cdots v_{2k}$ is a $C_{2k+1}$ in $G$, a contradiction.  Thus $w=v_j$ for some $j\ne i,i+1$.  
	
	Note that $v_j\ne i+2,i+3$.  Indeed if, say, $j=i+2$, then $v_1v_2\cdots v_i v_{i+2}\cdots v_{2k}$ would be a $C_{2k-1}$ in $G$, a contradiction.  A similar result holds if $j=i+3$.  Observe that $G$ contains the cycles $v_iv_{i+1}\cdots v_j$ and $v_{i+1}v_{i+2}\cdots v_j$.  One of these cycles must have odd parity with length at least 5 since $j\ne i+2,i+3$, a contradiction.  We conclude that $G$ contains no $C_{2k}$ with $k\ge 3$, proving the result.
\end{proof}

\begin{proof}[Proof of Proposition~\ref{P-3Alg}]
	$G^0$ is 1-good, so inductively assume that Mini has been able to play so that $G^{t-2}$ is 1-good for some even $t\ge 2$.  If $G^{t-1}$ is saturated then we are done, so assume this is not the case.  If  $G^{t-1}$ is 0-good, then Mini plays $e^t$ arbitrarily and $G^t$ will be 1-good.
	
	Now assume that $G^{t-1}$ is not 0-good.  That is, there exist edges $v_1x$ and $v_2x$ with $v_1\ne v_2$ such that neither of these edges are contained in triangles.   Mini then adds the edge $e^t=v_1v_2$, which we claim is a legal move.  If it were not, then there must exist a path $P$ of length $2k$ with $k\ge2$ from $v_1$ to $v_2$ in $G^{t-1}$.  If $x$ is not a vertex of $P$, then $G^{t-1}$ contains the cycle formed by taking $P$ and adding the edges $xv_1$ and $xv_2$, which is a $C_{2k+2}$.  Since inductively $G^{t-2}$ is 1-good, $G^{t-1}$ is 2-good, and hence does not contain such a $C_{2k+2}$ by Lemma~\ref{L-Good}. Thus $x$ must be a vertex of $P$.  
	
	Assume without loss of generality that $P$ does not use the edge $xv_1$.  Let $P_1$ denote the subpath of $P$ from $x$ to $v_1$ and let $k_1$ denote the length of $P_1$.  Note that $k_1\ne 1$ since $P$ does not use $xv_1$, and that $k_1\ne 2$ since this would imply that $xv_1$ is contained in a triangle.  Note that $G^{t-1}$ contains a $C_{k_1+1}$, namely by taking $P_1$ together with the edge $xv_1$. Thus $k_1\le 3$ by Lemma~\ref{L-Good}, so we conclude that $k_1=3$.
	
	Let $C=v_1abx$ be the 4-cycle formed from $P_1$ and $xv_1$.  If, say, $ab$ were contained in a triangle $abc$, then we must have $c=v_1$ or $c=x$, as otherwise $v_1acbx$ defines a $C_5$ in $G^{t-1}$.  But if $c=v_1$ or $x$, then $v_1x$ is contained in a triangle, a contradiction.  A similar analysis shows that no edge of $C$ is contained in a triangle.  This is only possible if $B(G^{t-1})$ consists of two edges of $C$ that are not both incident to $x$, as otherwise one of $ab$ and $v_1a$ would be contained in a triangle.  In particular, two of the edges $\{xv_1,xv_2,xb\}$ are not in $B(G^{t-1})$, and we conclude that at least one of these edges must be contained in a triangle.  But we have assumed that none of these edges are in triangles, a contradiction.  We conclude that $v_1v_2$ is a legal move to play.
	
	Note that at least one of the edges $xv_1$ and $xv_2$ must be in $B(G^{t-1})$, as otherwise $G^{t-1}- B(G^{t-1})$ would not have all good vertices (namely, $x$ would not be a good vertex).  Since $v_1x$, $v_2x$, and the new edge $v_1v_2$ are contained in a triangle of $G^t$, the set $B(G^t):=B(G^{t-1})\sm \{v_1x,v_2x\}$ shows that $G^t$ is 1-good as desired.
\end{proof}

It remains to bound how many edges $G^\infty$ will have after Mini uses the strategy of Proposition~\ref{P-3Alg}.

\begin{lem}\label{L-3Up}
	$\ex(n,\{C_5,C_6,\ldots\})\le 2n-2$.
\end{lem}
\begin{proof}
	The statement is trivially true for $n\le 4$, so assume we have proven the statement up to $n>4$. Let $G$ be an extremal $n$-vertex graph and assume that $e(G)\ge 2n-1$.  If $G$ contains a vertex $x$ with $d(x)\le 2$, then $G'=G-\{x\}$ is an $(n-1)$-vertex graph with $e(G')\ge 2(n-1)-1$.  By our inductive hypothesis, $G'$ contains a large cycle, and hence the same is true for $G$.  Thus we can assume that every vertex of $G$ has degree at least 3.
	
	We can assume that $G$ is connected, as adding an edge between two components of $G$ would increase $e(G)$ without creating any cycles. Let $T$ be a depth-first-search tree of $G$.  For any $x\in T$, let $x_1$ denote the parent of $x$ in $T$, and recursively define $x_i=(x_{i-1})_1$ for $i\ge 2$.  Observe that if $xy$ is an edge in $G$, then either $y=x_i$ or $x=y_i$ with $i=1,\ 2$, or 3.  Indeed, assume without loss of generality that $y$ was discovered before $x$ when constructing $T$.  Observe that the subtree of $T$ with $y$ as a root will contain every neighbor of $y$ that has not been discovered before $y$.  In particular, this subtree will contain $x$, and we will have $y=x_i$ where $i$ is the depth of $x$ in this subtree.  Further, we must have $i=1,\ 2,$ or 3, as otherwise $G$ would contain a $C_k$ with $k\ge 5$.
	
	Let $x$ be a vertex of maximum depth in $T$, which in particular means that $x$ is a leaf in $T$.  We wish to show that $\{x,x_1,x_2,x_3\}$ induces a $K_4$ in $G$ and that $d(x)=d(x_1)=d(x_2)=3$. To this end, we will say that a vertex $y$ has label $i$ with $i>1$ if $y_iy$ is an edge in $G$, and we let $S(y)$ denote the set of vertices $z\ne y$ with $z_1=y_1$.  Note that by our above argument, no vertex can have label $i>3$.
	
	Since we assumed $d(x)\ge 3$ and since $x$ only has one neighbor in $T$, $x$ must have both label 2 and 3.  We claim that $S(x)=\emptyset$.  Indeed, if there exists some $y\in S(x)$, then $y$ would also be a leaf (since $x$ is a vertex of maximum depth), so it would also have to have label 3, but then $G$ would contain the cycle $xx_2x_3yx_1$, a  contradiction.  Thus we must have $S(x)=\emptyset$.  Since $d(x_1)\ge 3$, and since $S(x)=\emptyset$, $x_1$ must have at least one of label 2 or 3, but if it had label 3 then $G$ would contain the cycle $xx_2x_3x_4x_1$, so $x_1$ only has label 2.  If $x_2$ had label 2, then $G$ would contain the cycle $xx_1x_2x_4x_3$, and a similar result holds if $x_2$ had label 3.  Thus $x_2$ does not have any label.  We claim that $S(x_1)=\emptyset$.  Indeed if we had $y\in S(x_1)$ with $y$ a leaf, then $y$ must have label 2 and $G$ would contain the cycle $yx_3xx_1x_2$.  Otherwise, $y$ would have a child $z$ which is a leaf (since it is at the same depth as $x$) and hence must have label 3, which means that $G$ contains the cycle $zx_3xx_1x_2y$.  Thus $S(x_1)=\emptyset$, proving our claims about the vertices $x,x_1,x_2,x_3$.
	
	Let $G'=G-\{x,x_1,x_2\}$.  By our above analysis, $G'$ is an $(n-3)$-vertex graph with $e(G')=e(G)-6\ge2(n-3)-1$, so by the induction hypothesis $G'$, and hence $G$, contains a large cycle, proving the desired bound.
\end{proof}
We note that the above bound is sharp, as can be seen by taking $G$ to be the graph obtained by taking $k$ disjoint triangles and then adding an additional vertex which is made adjacent to every other vertex. The above proof can easily be modified to characterize all extremal graphs, though we have no need for this here.

\begin{proof}[Proof of Theorem~\ref{T-Mod3}]
	Mini uses the strategy of Proposition~\ref{P-3Alg}.  This implies that $G^\infty$ is 2-good with no $C_{2k+1}$ for any $k\ge 2$, and hence contains no $C_k$ for any $k\ge 5$.  Lemma~\ref{L-3Up} then implies that $e(G^\infty)\le 2n-2$.
\end{proof}

\section{Concluding Remarks}\label{S-Con}

We claim that by analyzing the proof of Theorem~\ref{T-gen}, one can conclude that $\sat_g(\c{C}_5;n)$ and $\sat_g(\c{C}_7;n)$ are strictly less than $\floor{\quart n^2}$, which, as we mention in the beginning of Section~\ref{S-Up} is a non-trivial result.  We suspect that stronger bounds exist.

\begin{conj}
	For all $k\ge 1$ there exists a $c_k>0$ such that \[\sat_g(\c{C}_{2k+1};n)\le \left(\quart -c_k\right)n^2+o(n^2).\]
\end{conj}
In fact, we believe that a stronger statement is true. As a consequence of the bounds of Theorem~\ref{T-gen}, we know that $\sat_g(\c{C}_{2k+1};n)\le \sat_g(\c{C}_{2k'+1};n)$ when $k'$ is sufficiently larger than $k$ and $n$ is sufficiently large.  We conjecture that this remains true when $k'=k+1$.

\begin{conj}\label{C-Main}
	For all $k\ge 2$, \[\sat_g(\c{C}_{2k-1};n)\le \sat_g(\c{C}_{2k+1};n)\]
	for $n$ sufficiently large.
\end{conj}

Note that the bound $\sat_g(\{C_3\};n)\le \f{26}{121}n^2+o(n^2)$ of \cite{horn} together with Theorem~\ref{T-5} shows that the conjecture is true for $k=2$, and moreover that $\sat_g(\c{C}_3;n)\le \sat_g(\c{C}_{2k+1};n)$ for all $k\ge 2$ and $n$ sufficiently large.

Theorem~\ref{T-5} shows that $\sat_g(\c{C}_\infty\sm \{C_{2k+1}\};n)$ is quadratic for all $k\ge 3$. Theorem~\ref{T-Mod3} shows that $\sat_g(\c{C}_\infty\sm \{C_{3}\};n)$ is linear.   Given this, it is natural to ask about the order of magnitude of $\sat_g(\c{C}_\infty\sm \{C_{5}\};n)$.

\begin{quest}
	What is the order of magnitude of $\sat_g(\c{C}_\infty \sm \{C_5\};n)$?  In particular, is this value linear, quadratic, or something else?
\end{quest}

\section{Acknowledgments}
The author would like to thank Jacques Verstraete for suggesting this research topic, as well as his assistance with the general structure of the paper.  The author would also like to thank the anonymous referees, whose comments greatly improved the structure and readability of this paper.  This research was partially supported by NSF grant DMS-1800746.

\bibliographystyle{plain}
\bibliography{Hajnal}

\begin{thebibliography}{10}

\bibitem{horn}
Csaba Bir{\'o}, Paul Horn, and D~Jacob Wildstrom.
\newblock An upper bound on the extremal version of hajnal’s triangle-free
  game.
\newblock {\em Discrete Applied Mathematics}, 198:20--28, 2016.

\bibitem{westSurv}
James~M Carraher, William~B Kinnersley, Benjamin Reiniger, and Douglas~B West.
\newblock The game saturation number of a graph.
\newblock {\em Journal of Graph Theory}, 85(2):481--495, 2017.

\bibitem{westMatch}
Daniel~W Cranston, William~B Kinnersley, O~Suil, and Douglas~B West.
\newblock Game matching number of graphs.
\newblock {\em Discrete Applied Mathematics}, 161(13-14):1828--1836, 2013.

\bibitem{furedi}
Zolt{\'a}n F{\"u}redi, Dave Reimer, and A~Seress.
\newblock Hajnal's triangle-free game and extremal graph problems.
\newblock {\em Congressus Numerantium}, pages 123--123, 1991.

\bibitem{pralat}
Przemys{\l}aw Gordinowicz and Pawe{\l} Pra{\l}at.
\newblock The first player wins the one-colour triangle avoidance game on 16
  vertices.
\newblock {\em Discussiones Mathematicae Graph Theory}, 32(1):181--185, 2012.

\bibitem{hefetz}
Dan Hefetz, Michael Krivelevich, Alon Naor, and Milo{\v{s}} Stojakovi{\'c}.
\newblock On saturation games.
\newblock {\em European Journal of Combinatorics}, 51:315--335, 2016.

\bibitem{keusch}
Ralph Keusch.
\newblock Colorability saturation games.
\newblock {\em arXiv preprint arXiv:1606.09099}, 2016.

\bibitem{lee2}
Jonathan~D Lee and Ago-Erik Riet.
\newblock New $\mathcal{F}$-saturation games on directed graphs.
\newblock {\em arXiv preprint arXiv:1409.0565}, 2014.

\bibitem{lee1}
Jonathan~D Lee and Ago-Erik Riet.
\newblock F-saturation games.
\newblock {\em Discrete Mathematics}, 338(12):2356--2362, 2015.

\bibitem{patkos}
Bal{\'a}zs Patk{\'o}s and M{\'a}t{\'e} Vizer.
\newblock Game saturation of intersecting families.
\newblock {\em Central European Journal of Mathematics}, 12(9):1382--1389,
  2014.

\end{thebibliography}
\end{document}